\documentclass[12pt]{amsart}
\usepackage{amsfonts}
\usepackage{amssymb, amscd, amsthm}
\usepackage{latexsym}
\usepackage{multirow}
\usepackage{amsmath}
\usepackage[all]{xy}
\usepackage{mathrsfs}
\usepackage{mathtools}
\usepackage{amsbsy}
\usepackage{amsfonts}
\usepackage{mathrsfs}
\usepackage{verbatim}
\usepackage{version}
\usepackage{color}

\newtheorem{theorem}{Theorem}[section]
\newtheorem{lemma}[theorem]{Lemma}
\newtheorem{thm}[theorem]{Theorem}
\newtheorem{prop}[theorem]{Proposition}
\newtheorem{rem}[theorem]{Remark}
\newtheorem{coro}[theorem]{Corollary}

\newtheorem{con}[theorem]{Conjecture}

\newcommand{\ra}{\rightarrow}
\newcommand{\mo}{\mathcal{O}}
\newcommand{\mf}{\mathcal{F}}
\newcommand{\mg}{\mathcal{G}}
\newcommand{\ma}{\mathcal{A}}
\newcommand{\mb}{\mathcal{B}}
\newcommand{\me}{\mathcal{E}}
\newcommand{\mi}{\mathcal{I}}
\newcommand{\mj}{\mathcal{J}}
\newcommand{\mk}{\mathcal{K}}

\newcommand{\mh}{\mathcal{H}}
\newcommand{\mm}{\mathcal{M}}

\newcommand{\cl}{\mathcal{L}}

\newcommand{\rr}{\mathbf{R}}
\newcommand{\rl}{\mathbf{L}}

\newcommand{\tw}{\textbf{W}}
\newcommand{\mc}{\mathcal{C}}
\newcommand{\mr}{\mathcal{R}}

\newcommand{\E}{\mathscr{E}}
\newcommand{\F}{\mathscr{F}}

\newcommand{\ttb}{\mathtt{b}}

\newcommand{\Ext}{\operatorname{Ext}}

\newcommand{\Ker}{\operatorname{Ker}}
\newcommand{\Tor}{\operatorname{Tor}}
\newcommand{\Ima}{\operatorname{Im}}

\def\<{\langle}
\def\>{\rangle}

\newcommand{\ls}{|L|}
\newcommand{\p}{\mathbb{P}}
\newcommand{\bz}{\mathbb{Z}}
\newcommand{\bq}{\mathbb{Q}}
\newcommand{\bc}{\mathbb{C}}

\begin{document}
\fontsize{12pt}{14pt} \textwidth=14cm \textheight=21 cm
\numberwithin{equation}{section}
\title{On the perverse filtration of the moduli spaces of 1-dimensional sheaves on $\mathbb{P}^2$ and P=C conjecture.}
\author{Yao Yuan}
\thanks{The author is supported by NSFC 21022107.  }

\begin{abstract}Let $M(d,\chi)$ be the moduli space of semistable 1-dimensional sheaves supported at curves of degree $d$ on $\p^2$, with Euler characteristic $\chi$.  We have the Hilbert-Chow morphism $\pi: M(d,\chi)\ra |dH|$ sending each sheaf to its support.  We study the perverse filtration on $H^*(M(d,\chi),\bq)$ via map $\pi$, especially the ``P=C" conjecture posed by Kononov-Pi-Shen.  We show that ``P=C" conjecture holds for $H^{*\leq 4}(M(d,\chi),\bq)$ for any $d\geq 4$, $(d,\chi)=1$.  The main strategy is to relate $M(d,\chi)$ to the Hilbert scheme $S^{[n]}$ of $n$-points and transfer the problem to some properties on $H^*(S^{[n]},\bq)$.  We use induction on $n$ to achieve the desired properties.  Our proof involves some complicated calculations. 

~~~

 \end{abstract}

\maketitle
\tableofcontents
\section{Introduction.}

\subsection{Decomposition theorem and Perverse filtration.}
We work over the complex numbers $\bc$.

Let $f:X\ra Y$ be a proper morphism between two nonsingular quasi-projective varieties with $\dim X=n_X,\dim Y=n_Y$.  For simplicity, we assume moreover every fiber of $f$ has the same dimension $r=n_X-n_Y$.  By the well-known Decomposition theorem over $\bc$ (Th\'eor\`emes 6.25, 6.2.10 in \cite{BBD}) we have 
\[Rf_{*}\bq_X[n_X]\cong \bigoplus_i \prescript{p}{}{\mh}^{i}(Rf_{*}\bq_X[n_X])[-i],\]
where $\prescript{p}{}{\mh}^{i}(Rf_{*}\bq_X[n_X])$ are all semisimple perverse sheaves on $Y$.  Define
\[\prescript{p}{}{\tau}_{\leq i}(R\pi_{*}\bq_{X}[n_X])):=\bigoplus_{j\leq i} \prescript{p}{}{\mh}^{j}(Rf_{*}\bq_X[n_X])[-i]\]

For each $m\in \bz_{\geq 0}$ we have the \emph{perverse filtration} (with respect to $f$) $P_{\bullet}$ on the cohomology group $H^m(X,\bq)$ as follows
\[P_0H^m(X,\bq)\subset P_1H^m(X,\bq)\subset\cdots\subset P_{2r}H^m(X,\bq)=H^m(X,\bq),\]
where 
$$P_iH^m(X,\bq):=\Ima\left\{ H^{m-n_X}(Y,\prescript{p}{}{\tau}_{\leq i}(R\pi_{*}\bq_{X}[n_X]))\ra H^m(X,\bq)\right\}.$$
Hence we naturally have
\[P_iH^m(X,\bq)\cong\bigoplus_{j\leq i}\mathbb{H}^{m-n_X-j}(Y,\prescript{p}{}{\mh}^{i}(Rf_{*}\bq_X[n_X])).\]

We say that a class $\gamma\in H^*(X,\bq)$ has \emph{perversity} $k$ if $\gamma\in P_kH^*(X,\bq)\setminus P_{k-1}H^*(X,\bq)$.  Perverse filtration in general is very complicated to describe since we know very few on the perverse truncation functor.  However if $Y$ is projective, denote by $\eta$ the pullback of some ample divisor on $Y$ to $X$, then we have the following description due to de Cataldo-Migliorini in \cite[Proposition 5.2.4]{CM}.
\begin{prop}\label{despintro}We have 
$$P_kH^m(X,\bq)=\sum_{i\geq 1}\left(\Ker(\eta^{n_Y+k+i-m})\cap\Ima(\eta^{i-1})\right)\cap H^m(X,\bq).$$
In particular for any $\gamma\in H^*(X,\bq)$ of perversity $k$, $\eta^n\gamma\in H^{*+2n}(M(d,\chi),\bq)$ is either zero or of perversity $k$ for any $n\in\bz_{>0}$.
\end{prop}
\subsection{Motivation and the main result.}
We consider the following two kinds of moduli space $\mm^H_S(L,\chi)$ and $\mm_{C_0,D}^{Higgs}(n,\chi)$ as follows.
\begin{enumerate}
\item[(A)] $\mm^H_X(L,\chi)$ parametrizes semistable sheaves $\mf$ with respect to the polarization $H$ on a projective surface $S$, which are supported on curves in the linear system $\ls$ and of Euler characteristic $\chi$.  We have a Hilbert-Chow morphism
\begin{equation}\label{HCmap}\pi_{L,\chi}:\mm^H_S(L,\chi)\ra \ls,~~\mf\mapsto \text{supp}(\mf).\end{equation}

\item[(B)] $\mm_{C_0,D}^{Higgs}(n,\chi)$ parametrizes semistable Higgs bundles $(\me,\Theta)$ with respect to the effective divisor $D$ on a smooth curve $C_0$ with $\me$ rank $n$ and Euler characteristic $\chi$.  We have the Hitchin fibration
\begin{equation}\label{HF}\pi_{D,\chi}:\mm_{C_0,D}^{Higgs}(n,\chi)\ra \bigoplus_{i=1}^n H^0(C_0,\mo_{C_0}(iD)),\end{equation}
sending $(\me,\Theta)$ to the coefficients $(-\text{tr}(\Theta),\cdots,(-1)^n\det(\Theta))$ of the characteristic polynomials $\text{char}(\Theta)$ of the morphism $\Theta:\me\ra\me(D)$.
\end{enumerate}
Both $\pi_{L,\chi}$ and $\pi_{D,\chi}$ are proper maps.  Under suitable assumptions they are also maps between nonsingular quasiprojective varieties with fibers of the same dimensions.  

When the fibration in (A) or (B) is Lagragian, the perverse filtration can have a nice coincidence with some weight filtration: we have the so called ``P=W" property conjectured in \cite{CHM}, proved in \cite{CMS} for case (A) with $S$ an abelian surface, in \cite{HMMS} and \cite{MS} independently for case (B) with $D$ the canonical divisor of $C_0$.  In particular, in these cases the perverse filtration is \emph{multiplicative}, i.e. for any $k_1,k_2,m_1,m_2$
\begin{equation}P_{k_1}H^{m_1}(X,\bq)\cdot P_{k_2}H^{m_2}(X,\bq)\subset P_{k_1+k_2}H^{m_1+m_2}(X,\bq),
\end{equation}
where $\cdot$ denotes the cup product on $H^*(X,\bq)$.

Now in (A) let $S=\p^2$, then the Hilbert-Chow morphism is not Lagrangian.  It is interesting to ask in this case whether there is any analog of ``P=W" phenomenon, and whether the perverse filtration is multiplicative.  Let $H$ be the hyperplane class in $\p^2$.  Let $M(d,\chi)$ be the moduli space of stable sheaves supported at curves in $|dH|$ and with Euler characteristic $\chi$.  $M(d,\chi)$ is smooth since $\p^2$ is Fano.  We also ask $M(d,\chi)$ to be projective $\Leftrightarrow (d,\chi)=1$. 

By Pi-Shen's work in \cite{PS}, a minimal set of generators of $H^*(M(d,\chi),\bq)$ can be chosen as 
\begin{equation}\label{gentintro}\Sigma:=\{c_0(2),c_2(0),c_k(0),c_{k-1}(1),c_{k-2}(2), k\in {3,\cdots,d-1}\}.\end{equation}

In \cite{KPS}, Kononov-Pi-Shen call the subspace $$H^{*\leq 2d-4}(M(d,\chi),\bq)\subset H^*(M(d,\chi),\bq)$$ the \emph{free part} of the cohomology and they define the \emph{Chern filtration} of the free part 
\[C_0H^{*\leq 2d-4}(M(d,\chi),\bq)\subset C_1H^{*\leq 2d-4}(M(d,\chi),\bq)\subset\cdots\subset H^{*\leq 2d-4}(M(d,\chi),\bq),\]
where the $k$-th piece $C_kH^{*\leq 2d-4}$ is defined as the span of all monomials
$$\prod_{i=1}^sc_{k_i}(j_i)\in H^{*\leq 2d-4}(M(d,\chi),\bq),\text{ and }\sum_{i=1}^s k_i\leq k,$$
where $c_{k_i}(j_i)$ are generators given in (\ref{gentintro}).

Kononov-Pi-Shen pose the following conjecture which is also called ``P=C" conjecture.
\begin{con}[Conjecture 0.3 in \cite{KPS}]\label{kpscon} For $d\geq 3$, we have
\[P_kH^{*\leq 2d-4}(M(d,\chi),\bq)=C_kH^{*\leq 2d-4}(M(d,\chi),\bq).\]
\end{con}
Conjecture \ref{kpscon} can be viewed as an analog of ``P=W" conjecture for Lagragian fibration cases.  Our result is as follows.  
\begin{thm}[Corollary \ref{mainco}]\label{mainintro} For any $d\geq 4$ and $\chi$ co-prime to $d$, we have
\[P_kH^{*\leq 4}(M(d,\chi),\bq)=C_kH^{*\leq 4}(M(d,\chi),\bq).\]
\end{thm}

Our strategy to prove Theorem \ref{mainintro} is to relate the moduli space $M(d,\chi)$ to the Hilbert scheme $S^{[n]}$ of $n$-points on $S$, then we transfer the perversity of those generators $c_k(j)$ to some intersection properties of some classes on $S^{[n]}$, and finally we do a lot of computations to show those properties. 
\begin{rem}In principle, following our strategy one can compute the perversity of more elements in $H^{*\leq 2d-4}(M(d,\chi),\bq)$.  But it may cost one huge amount of computations.  Hence at this moment we just finish the case $H^{*\leq 4}(M(d,\chi),\bq)$ and leave others to the future study.  
\end{rem}

\subsection{The plan of the paper.}
We will introduce some notations and conventions in next subsection.  In \S 2 we review in more details on the cohomology theory of $M(d,\chi)$ and the set-up of the ``P=C" conjecture.  In \S 3.1 and \S 3.2 we built the relation between $M(d,\chi)$ and $S^{[n]}$, and then transfer our problem to $S^{[n]}$.  In \S3.3 and \S 3.4 we review the definition of incidence varieties for Hilbert schemes of points, and then deduce some recursion formulas which will play a key role in our proof.  In \S 4 we will carry on explicit computations where two technical lemmas have their proofs moved to the appendix.  
\subsection{Notations \& Conventions.}\label{NC}
\begin{itemize}
\item For any projective smooth scheme $X$, denote by $A^*(X)=\oplus_{r\geq 0}A^r(X)$ the Chow ring (with $\bq$-coefficients) of $X$ and denote by $K(X)$ the Grothendieck group of coherent sheaves over $X$.  For any closed subscheme $Y\subset X$ (coherent sheaf $\mf$ over $X$, resp.), denote by $[Y]$ ($[\mf]$, resp.) its class in $A^*(X)$ ($K(X)$, resp.).

\item
Except otherwise stated, we always let $S=\p^2_{\bc}$ with $K_S$ the canonical divisor class and $H$ the hyperplane class.  Hence $K_S=-3H$.  
\item For any map $f:X\ra Y$, define $f^S:=id_S\times f: S\times X\ra S\times Y$. 





%
\item We use the same letter for a line bundle and its divisor class, for instance, $L_1\otimes L_2$ is the tensor of two line bundles, and $L_1\cdot L_2$ is the intersection of their divisor classes.  Define $L_1^k=\underbrace{L_1\circ\cdots\circ L_1}_{k~times}$.

\item For any map $f:X\ra Y$, we sometimes (if there is no confusion) use the same letter for a class in $H^*(Y)$ and its pullback to $H^*(X,\bq)$.




\end{itemize}

\subsection{Acknowledgements.} 
I would like to thank Kononov-Pi-Shen for their conjecture which motivates this paper.
I thank the referees in advance for their attention to read the paper.

\section{The cohomology ring of $M(d,\chi)$ and the P=C conjecture.}
\subsection{Perverse filtration on the cohomology groups.}  Let $\chi$ be an integer.  Let $M(L,\chi)^{ss}$ be the moduli space of semistable 1-dimensional sheaves with determinant $L$ and Euler characteristic $\chi$ over a projective surface $S$.  All (S-equivalence classes of) sheaves in $M(L,\chi)^{ss}$ are supported on curves in the linear system $\ls$.  Let $M(L,\chi)\subset M(L,\chi)^{ss}$ parametrize stable sheaves in $M(L,\chi)^{ss}$.

We have the Hilbert-Chow morphism
\begin{equation}\label{hcm}\pi^{ss}:M(L,\chi)^{ss}\ra\ls,~\mf\mapsto \text{supp}\mf.\end{equation}
The map $\pi^{ss}$ in (\ref{hcm}) is not only a set-theoretic map but also a morphism of algebraic schemes (see e.g. \cite[Proposition 3.0.2]{Yuan1}).  The fiber of $\pi^{ss}$ over an integral curve $C$ is isomorphic to the (compactified) Jacobian of $C$ while the fibers over non-integral curves can have more than one irreducible component.  Denote by \begin{equation}\label{hcms}\pi:M(L,\chi)\ra\ls\end{equation} the Hilbert-Chow morphism restricted to $M(L,\chi)$.  For $X$ Fano, $M(L,\chi)$ is smooth.  We also have the following result.
\begin{prop}[Corollary 1.3 in \cite{Yuan9}]\label{dimest}If $S$ is Fano or with $K_S$ trivial, and if moreover $\ls$ contains an integral curve, then all the fibers of the Hilbert-Chow morphism $\pi$ in (\ref{hcms}) have dimension $g_L$, where $g_{L}$ is the arithmetic genus of any curve in $\ls$. 
\end{prop}

Let $S$ be any Fano surface and let $M(L,\chi)=M(L,\chi)^{ss}=:M$, then $\pi$ is a proper map between two nonsingular projective varieties and has all fibers of dimension $g_L$.  Define $\mathtt{b}:=\dim\ls=n_M-g_L$.  For each $m\in \bz_{\geq 0}$ we have the \emph{perverse filtration} $P_{\bullet}$ with respect to $\pi$ on the cohomology group $H^m(M,\bq)$ as follows
\[P_0H^m(M,\bq)\subset P_1H^m(M,\bq)\subset\cdots\subset P_{2g_L}H^m(M,\bq)=H^m(M,\bq),\]
where 
$$P_iH^m(M,\bq):=\Ima\left\{ H^{m-n_M}(\ls,\prescript{p}{}{\tau}_{\leq i}(R\pi_{*}\bq_{M}[n_M]))\ra H^m(M,\bq)\right\}.$$
We say that a class $\gamma\in H^*(M,\bq)$ has \emph{perversity} $k$ if $\gamma\in P_kH^*(M,\bq)\setminus P_{k-1}H^*(M,\bq)$.

Denote by $\xi$ both the hyperplane class on $\ls$ and its pullback to $M(d,\chi)$ via the Hilbert-Chow morphism $\pi$.  Then the following proposition follows from Proposition \ref{despintro}
\begin{prop}\label{desp}We have 
$$P_kH^m(M(d,\chi),\bq)=\sum_{i\geq 1}\left(\Ker(\xi^{\ttb+k+i-m})\cap\Ima(\xi^{i-1})\right)\cap H^m(M(d,\chi),\bq).$$
In particular for any $\gamma\in H^*(M(d,\chi),\bq)$ of perversity $k$, $\xi^n\gamma\in H^{*+2i}(M(d,\chi),\bq)$ is either zero or of perversity $k$ for any $n\in\bz_{>0}$.
\end{prop}

\subsection{Generators for the cohomology ring.}
Let $S=\p^2$ and $L=dH$ with $d\in\bz_{>0}$.  If $(d,\chi)=1$, then $M(L,\chi)^{ss}=M(L,\chi)=:M(d,\chi)$ and it is a fine moduli space.    
Denote by $\E$ a universal sheaf over $S\times M(d,\chi)$.  For every Chern class $c_i(\E)$ of $\E$, we have the K\"unneth decomposition
\[c_i(\E)=H^0\otimes e_i(2)+H\otimes e_i(1)+H^2\otimes e_i(0)\]
with $e_i(j)\in A^{i+j-2}(M(d,\chi))$.  Notice that $H^*(\p^2,\bz)\cong \bz[H]/(H^3)$. 

By \cite[Theorem 1 and Theorem 2]{mark},  we have the cohomology ring $H^*(M(d,\chi),\bz)$ is torsion-free and isomorphic to the integral Chow ring $A^*(M(d,\chi))_{\bz}$.  
Moreover $H^*(M(d,\chi),\bz)$ is generated by the K\"unneth factors $e_i(j)$ of all the Chern classes of $\E$ (see \cite[Proposition]{Bea} or \cite[Proposition 12]{mark}).

In \cite{PS}, W. Pi and J. Shen studied the relation between those generators $e_i(j)$ and they proved that a minimal set of generators can be chosen as 
\begin{equation}\label{gent}\Sigma:=\{c_0(2),c_2(0),c_k(0),c_{k-1}(1),c_{k-2}(2), k\in {3,\cdots,d-1}\}.\end{equation}
The tautological class $c_k(j)\in A^{k+j-1}(M(d,\chi))$ is defined as the K\"unneth factor of the degree $k+1$ component ch$_{k+1}^{\alpha}(\E)$ of the twisted Chern character ch$^{\alpha}(\E):=\text{ch}(\E)\cdot \text{exp}(\alpha)$.  
In other words,
\begin{equation}\label{defckj}c_k(j):=\int_{H^j}(\text{ch}(\E)\cdot \text{exp}(\alpha))_{k+1}=\pi_{M*}(\pi_S^*H^j\cdot (\text{ch}(\E)\cdot \text{exp}(\alpha))_{k+1}),\end{equation}
where $\pi_M,\pi_S$ are two projections from $M(d,\chi)\times S$ to $M(d,\chi)$ and $S$ respectively.

The element $\alpha$ is defined as follows 
\begin{equation}\label{defal}\alpha:=\left(\frac32-\frac{\chi}d\right)\cdot \pi_S^*H-\frac1d\left(\left(\frac32-\frac{\chi}d\right)\pi_M^*\xi+\pi_M^*\left(\int_{H}\text{ch}_{2}(\E)\right)\right)\in A^1(S\times M(d,\chi)).\end{equation}

Notice that $\E$ is a torsion sheaf supported at a codimension 1 subscheme with class $d\pi_S^*H+\pi_M^*\xi$ on $S\times M(d,\chi)$.  Hence $\text{ch}(\E)_0=0$ and $\text{ch}(\E)_1=d\pi_S^*H+\pi_M^*\xi$.  By definition in (\ref{defckj}) one can see that $c_0(0)=0$, $c_0(1)=d\cdot\mathbf{1}$, $c_1(0)=c_1(1)=0$.  It also easy to compute that $c_0(2)=\xi$ and finally by \cite[Proposition 1.3(c)]{PS} $c_2(0)$ is a relatively ample class with respect to $\pi$.

\begin{thm}[Theorem 0.2 in \cite{PS}, Theorem 1.2 in \cite{Yuan10} ]\label{sgen2}Let $\chi$ be coprime to $d$.  Then the $3d-7$ generators 
$$c_0(2),c_2(0),c_k(0),c_{k-1}(1),c_{k-2}(2), k\in {3,\cdots,d-1}$$ of $A^*(M(d,\chi))\cong H^*(M(d,\chi),\mathbb{Z})$ given in (\ref{gent}) have no relation in $A^i(M(d,\chi)),i\leq d-2$ for $d\geq 3$, have no relation in $A^i(M(d,\chi)),i\leq d-1$ and have 3 linearly independent relations in $A^d(M(d,\chi))$ for $d\geq 5$. 
\end{thm}

\subsection{The P=C conjecture} In \cite{KPS}, Kononov-Pi-Shen call the subspace $$H^{*\leq 2d-4}(M(d,\chi),\bq)\subset H^*(M(d,\chi),\bq)$$ the \emph{free part} of the cohomology and they define the \emph{Chern filtration} of the free part 
\[C_0H^{*\leq 2d-4}(M(d,\chi),\bq)\subset C_1H^{*\leq 2d-4}(M(d,\chi),\bq)\subset\cdots\subset H^{*\leq 2d-4}(M(d,\chi),\bq),\]
where the $k$-th piece $C_kH^{*\leq 2d-4}$ is defined as the span of all monomials
$$\prod_{i=1}^sc_{k_i}(j_i)\in H^{*\leq 2d-4}(M(d,\chi),\bq),\text{ with } c_{k_i}(j_i)\in\Sigma\text{ and }\sum_{i=1}^s k_i\leq k.\footnote{We modify the definition a bit in order to exclude $c_1(0),c_1(1)$}$$
\begin{rem}The Chern filtration makes sense on the free part since every element in the free part can be uniquely expressed as a polynomial in generators in $\Sigma$.
\end{rem}
Kononov-Pi-Shen pose the following conjecture which is also called ``P=C" conjecture.
\begin{con}[Conjecture 0.3 in \cite{KPS}] For $d\geq 3$, we have
\[P_kH^{*\leq 2d-4}(M(d,\chi),\bq)=C_kH^{*\leq 2d-4}(M(d,\chi),\bq).\]
\end{con}
\begin{rem}\label{c0c2}Since $c_0(2)=\xi$ and $c_2(0)$ is relative ample, we automatically have $c_0(2)^n$ has perversity 0 for $n<\ttb$ and $c_2(0)^n$ has perversity $2n$ for $n<g_L=\frac{(d-1)(d-2)}2$. 
\end{rem}
The choice of $\alpha$ in (\ref{defal}) is a bit complicated.  The following two lemmas allow us to replace it by a simpler element.
\begin{lemma}\label{newal1}Let $\alpha'\in A^1(S\times M(d,\chi))$ such that $\alpha'-\alpha=a\pi_S^*H+b\pi_M^* \xi$ for some $a,b\in\bq$.  Define
\[c'_k(j):=\int_{H^j}(\emph{ch}(\E)\cdot \emph{exp}(\alpha'))_{k+1}\]
and 
\[\Sigma':=\{c'_0(2),c'_2(0),c'_k(0),c'_{k-1}(1),c'_{k-2}(2), k\in {3,\cdots,d-1}\}.\]
Then $\Sigma'$ also provides a minimal set of generators of $A^*(M(d,\chi))$, Theorem \ref{sgen2} also holds for generators in $\Sigma'$.\end{lemma}
\begin{proof}We have
\[\text{ch}(\E)\cdot \text{exp}(\alpha')=\text{ch}(\E)\cdot \text{exp}(\alpha)\cdot\text{exp}(a\pi_S^*H+b\pi_M^*\xi).\]
Hence
\[(\text{ch}(\E)\cdot \text{exp}(\alpha'))_{k+1}=\sum_{i=0}^{k+1}\frac{1}{i!}(a\pi_S^*H+b\pi_M^*\xi)^i (\text{ch}(\E)\cdot \text{exp}(\alpha))_{k+1-i}.\]
Notice that $\text{ch}(\E)_0=0$ and $\text{ch}(\E)_1=d\pi_S^*H+\pi_M^*\xi$. By a direct computation we have 
\begin{eqnarray}\label{comal}c'_k(2)&=&c_k(2)+\sum_{i=1}^{k} \frac{b^i}{i!}\xi^ic_{k-i}(2),\nonumber\\
c'_k(1)&=&c_k(1)+\sum_{i=1}^{k} \frac{b^i}{i!}\xi^ic_{k-i}(1)+\sum_{i=1}^{k}\frac{ab^{i-1}}{(i-1)!}\xi^{i-1}c_{k-i}(2),\nonumber\\
c'_k(0)&=&c_k(0)+\sum_{i=1}^{k} \frac{b^i}{i!}\xi^ic_{k-i}(0)+\sum_{i=1}^{k}\frac{ab^{i-1}}{(i-1)!}\xi^{i-1}c_{k-i}(1)\nonumber\\
&&+\sum_{i=2}^{k}\frac{a^2b^{i-2}}{2(i-2)!}\xi^{i-2}c_{k-i}(2). \end{eqnarray}
In particular we have $c'_0(j)=c_0(j)$ for $j=0,1,2$, $c'_1(0)=c_1(0)+ac_0(1)=ad\mathbf{1}$, and $c'_1(1)=c_1(1)+b\xi c_0(1)+ac_0(2)=(bd+a)\xi=(bd+a)c'_0(2)$.  The lemma follows by a direct observation.
\end{proof}
Since Theorem \ref{sgen2} applies to $\Sigma'$, we can define another Chern filtration $C'_{\bullet}$ on the free part $H^{*\leq 2d-4}(M(d,\chi),\bq)$ analogously with $\alpha$ replaced by $\alpha'$.  
\begin{lemma}\label{newal2}Let $\alpha$ and $\alpha'$ be as in Lemma \ref{newal1}.  Then for any $0\leq t\leq 2d-4$ the following two statements are equivalent
\begin{itemize}
\item[(1)]$P_kH^{*\leq t}(M(d,\chi),\bq)=C_kH^{*\leq t}(M(d,\chi),\bq),$
\item[(2)]$P_kH^{*\leq t}(M(d,\chi),\bq)=C'_kH^{*\leq t}(M(d,\chi),\bq).$
\end{itemize}
\end{lemma}
\begin{proof}By (\ref{comal}) and Proposition \ref{desp}, one can show the lemma by induction on $t$.  The argument is standard and we left it to the readers.
\end{proof}

From now on we redefine $\alpha$ as follows 
\begin{equation}\label{defal2}\alpha:=-\frac1d\left(\pi_M^*\left(\int_{H}\text{ch}_{2}(\E)\right)\right)\in A^1(S\times M(d,\chi)).\end{equation}
Define  
\begin{equation}\label{defc2}\text{ch}^{\alpha}(\E):=\text{ch}(\E)\cdot \text{exp}(\alpha)  \text{ and }c_k(j):=\int_{H^j}(\text{ch}^{\alpha}(\E))_{k+1}.\end{equation}
We also have the set 
\begin{equation}\label{gent2}\Sigma:=\{c_0(2),c_2(0),c_k(0),c_{k-1}(1),c_{k-2}(2), k\in {3,\cdots,d-1}\},\end{equation}
is a minimal set of generators, and $c_1(0)=(\chi-\frac{3d}2)\mathbf{1},c_1(1)=0$.  
On the ``P=C" conjecture we have the following result.
\begin{thm}\label{main1}For any $d\geq 5$ and $\chi$ co-prime to $d$, we have
\begin{itemize}\item[(1)]$c_1(2)$ has perversity 1;
\item[(2)]$c_2(1)$ has perversity 2;
\item[(3)]$c_3(0)$ has perversity 3.
\end{itemize}
\end{thm}

Theorem \ref{main1} together with Remark \ref{c0c2} and the result in \cite[\S2.4]{KPS} implies the following corollary.
\begin{coro}\label{mainco}For any $d\geq 4$ and $\chi$ co-prime to $d$, we have
\[P_kH^{*\leq 4}(M(d,\chi),\bq)=C_kH^{*\leq 4}(M(d,\chi),\bq).\] 
\end{coro}

\begin{rem}\label{chchi}For any $\chi'\equiv \chi ~(d)$, there is a natural isomorphism $M(d,\chi')\cong M(d,\chi)$.  Let $\chi_1=\chi+md$ for any $m\in\bz$, then $M(d,\chi)$ can be viewed as the moduli space $M(d,\chi_1)$ with $\E_1:=\E\otimes q^*H^{\otimes m}$ a universal sheaf.  Since
\[\alpha_1:=-\frac1d\left(\pi_M^*\left(\int_{H}\emph{ch}_{2}(\E_1)\right)\right)=\alpha-\frac md\pi_M^*\xi,\]
we have
\[\emph{ch}^{\alpha_1}(\E_1)=\emph{ch}^{\alpha}(\E)\otimes\emph{exp}(-\frac md\pi_M^*\xi+m\pi_S^*H).\]
Therefore by Lemma \ref{newal1} and Lemma \ref{newal2}, it is enough to show Theorem \ref{main1} for some $M(d,\chi')$ with $\chi'\equiv\chi~(d)$.
\end{rem}

Our strategy to prove Theorem \ref{main1} is to relate $M(d,\chi)$ to the Hilbert scheme $S^{[n]}$ of $n$-points, more precisely, to the relative Hilbert scheme $\mc^n_{\ls}$ of $n$-points over $\ls$.

\section{Some properties on the Hilbert schemes of points.}
\subsection{The relative Hilbert scheme of points over a linear system.} 
Let $S^{[n]}$  be the Hilbert scheme of $n$ points on $S$.  Let $\mi_n$ be the ideal sheaf of the universal family $\Sigma_n\subset S^{[n]}\times S$.

We define a group homomorphism as follows
\begin{eqnarray}\label{ghs}K(S)&\ra&K(S^{[n]})\nonumber\\
\alpha&\mapsto&\alpha^{[n]}:=R(p_n)_*(q_n^*\alpha\cdot\mo_{\Sigma_n}),
\end{eqnarray}
where $\mf\cdot\mg:=\sum_{i\geq0}(-1)^i\Tor^i(\mf,\mg)$ and $Rf_*\mf=\sum_{i\geq 0}(-1)^iR^if_*\mf$.  If $\alpha$ is the class of a vector bundle $\me$ of rank $r$ over $S$, then $\alpha^{[n]}$ is the class of $(p_n)_*(q_n^*\me|_{\Sigma_n})$ which is a vector bundle of rank $nr$ over $S^{[n]}$. 

Let $L=dH$ with $d\geq 2$.  Let $\ls$ be the linear system of $L$ and let $\ls^{int}$ ($\ls^{sm}$, resp.) be the open subscheme of $\ls$ parametrizing integral (smooth, resp.) curves.  One can see that $\ls\setminus\ls^{int}$ is of codimension $d-1$ in $\ls$.  Therefore we can choose a generic $d-2$ dimensional projective space $\p^{d-2}\cong D\subset \ls$ such that $D\subset \ls^{int}$.  Recall we have $\mathtt{b}:=\dim\ls$.

Denote by $\mc_{\ls}^n$ the relative Hilbert scheme of $n$-points over the linear system $\ls$.  In other words, $\mc_{\ls}^n$ parametrizes all pairs $(Z,C)\in S^{[n]}\times\ls$ such that $Z \subset C\subset S$.  If $n=1$, then $\mc_{\ls}^n$ is the universal curve in $S\times\ls$ which we denote by $\mc_{\ls}$.

We can view $\mc_{\ls}^n$ as a closed subscheme of $S^{[n]}\times \ls$.  Denote by $\mc_{\ls^{int}}^n$ ($\mc_{\ls^{sm}}^n$, $\mc_{D}^n$, resp.) the pullback of $\mc_{\ls}^n$ via $\ls^{int}\hookrightarrow\ls$ ($\ls^{sm}$, $D\hookrightarrow\ls$, resp.).  By abuse of notation, we denote by $\xi$ ($P_i^{[n]}$, resp.) both the hyperplane class (the $i$-th Chern class of the rank $n$ vector bundle $[L]^{[n]}$, resp.) on $\ls$ ($S^{[n]}$, resp.) and its pullback to $S^{[n]}\times \ls$.   Define
\begin{equation}\label{defpx}P^{[n]}(\xi):=P_n^{[n]}+P_{n-1}^{[n]}\xi+\cdots+P_1^{[n]}\xi^{n-1}+\xi^{n}\in A^n(S^{[n]}\times \ls).\end{equation}
\begin{prop}\label{clofc}In $A^*(S^{[n]}\times\ls^{int})$ we have 
$$[\mc_{\ls^{int}}^n]=P^{[n]}(\xi)|_{S^{[n]}\times\ls^{int}}.$$
In particular in $A^*(S^{[n]}\times\ls)$ we have
$$[\mc_{D}^n]=\xi^{\ttb-(d-2)}\cdot(P^{[n]}(\xi)).$$
\end{prop}
\begin{proof}By definition $[L]^{[n]}$ is the rank $n$ tautological vector bundle over $S^{[n]}$ via the assignment
\[Z\mapsto H^0(L\otimes \mo_Z).\]
We have the evaluation map $H^0(L)\xrightarrow{v_Z}H^0(L\otimes\mo_Z)$, and $\mc_{\ls}^n\subset S^{[n]}\times\p(H^0(L))$ consists of all pairs $(Z,[\bc s])\in S^{[n]}\times \p(H^0(L))$ such that $s\in \Ker(v_Z)\setminus{0}$.

Therefore, $\mc_{\ls}^n$ can be embedded into $\p(\mo_{S^{[n]}}\otimes H^0(L)^{\vee})\cong S^{[n]}\times \ls$ as the zero locus of the homomorphism
\[\pi_s^*([L]^{[n]})^{\vee}\ra \pi_s^{*}(\mo_{S^{[n]}}\otimes H^0(L)^{\vee})\ra \mo_{\pi_s}(1),\]
where $\pi_s:\p(\mo_{S^{[n]}}\otimes H^0(L)^{\vee})\ra S^{[n]}$ is the projection.  It is easy to see that $\mo_{\pi_s}(1)\cong \pi_l^*\mo_{\ls}(1)\cong\mo_{S^{[n]}\times\ls}(\xi)$ with $\pi_l:\p(\mo_{S^{[n]}}\otimes H^0(L)^{\vee})\cong S^{[n]}\times \ls\ra\ls$ the projection.

Since the map $\mc_{\ls^{int}}^n\ra \ls^{int}$ has all its fibers of dimension $n$ (see e.g. p214, last paragraph in \cite{Re}), it contains $\mc_{\ls^{sm}}^{n}$ as a smooth dense subset and therefore $\mc_{\ls^{int}}^n$ is irreducible of dimension $n+\ttb$.   Therefore in $A^*(S^{[n]}\times \ls^{int})$
\begin{eqnarray}[\mc_{\ls^{int}}^n]&=&c_n(\mo_{\pi}(1)\otimes \pi^*([L]^{[n]}))|_{S^{[n]}\times\ls^{int}}\nonumber\\
&=&P^{[n]}(\xi)|_{S^{[n]}\times\ls^{int}}.\nonumber\end{eqnarray}

Since $D\subset\ls^{int}$ and $[D]=\xi^{\ttb-(d-2)}$ in $A^*(\ls)$, the statement for $[\mc_D^n]$ follows straightforward.  
\end{proof}

\begin{rem}At this moment we can not prove that $\mc_{\ls}^{n}$ is irreducible of dimension $n+\ttb$.  However if it is true, then we would have in $A^*(S^{[n]}\times \ls)$
\[ [\mc_{\ls}^n]=P^{[n]}(\xi).\]
It will be interesting to check in the future whether the fibers of the projection $\mc_{\ls}^{n}\ra \ls$ over non-integral curves are also of dimension $n$.\end{rem}

\subsection{To relate $M(d,\chi)$ to $\mc_{\ls}^n$.} For any map $f:X\ra Y$, define $f^S:=id_S\times f: S\times X\ra S\times Y$. 
We have the following commutative diagram
\begin{equation}\label{csl}\xymatrix@C=1.2cm{&S\times\mc_{\ls}^n\ar[ld]_{p_c}\ar[r]^{\imath^S}\ar@/^2pc/[rr]^{q_1}&S\times S^{[n]}\times \ls\ar[d]^{\pi_s^S}\ar[ld]_{p_{23}}\ar[r]^{\quad q}\ar[rd]^{\qquad q }&S\\ \mc_{\ls}^n\ar[r]^{\imath}\ar[d]_{\varphi_l=\pi_l\circ\imath}\ar[rd]^{\varphi_s}&S^{[n]}\times \ls\ar[d]_{\pi_s}\ar[rd]_{p_n\qquad}&S\times S^{[n]} \ar[ld]^{\qquad \pi_l}\ar[ru]_{\pi_l^S}&S\times\ls\ar[ld]_{p_l}  \\ \ls& S^{[n]}&\ls&}
\end{equation}
Over $S\times S^{[n]}\times\ls$ we have the diagram
\begin{equation}\label{tuf}\xymatrix@C=0.6cm@R=0.6cm{&&0\ar[d]&&\\ &&(\pi_s^S)^*\mi_n\otimes q^*L\ar[d]&&\\ 0\ar[r]&(\pi^S_l)^*p_{l}^*\mo_{\ls}(-\xi)\ar[r]^{\qquad u}&q^*L\ar[r]\ar[d]^{v}&(\pi^S_l)^*\mo_{\mc_{\ls}}\otimes q^*L\ar[r]&0\\ &&(\pi^S_s)\mo_{\Sigma_n}\otimes q^*L\ar[d]&&\\ &&0&& },\end{equation}
where the map $v\circ u$ is zero over $S\times \mc_{\ls}^n$.   Therefore over $S\times\mc_{\ls}^n$ we have a non-zero hence injective map $$p_c^*\varphi_l^*\mo_{\ls}(-\xi)=(\imath^S)^*(\pi^S_l)^*p_{l}^*\mo_{\ls}(-\xi)\hookrightarrow (\imath^S)^*(\pi_s^S)^*\mi_n\otimes q^*L$$ and thus an exact sequence as follows.
\begin{equation}\label{ufrh}0\ra (\imath^S)^*(\pi^S_l)^*p_{l}^*\mo_{\ls}(-\xi)\xrightarrow{\tilde{v}} (\imath^S)^*(\pi_s^S)^*\mi_n\otimes q^*L\ra\F\ra0.
\end{equation}
The map $\tilde{v}$ restricted to every fiber over $(Z,C)\in\mc_{\ls}^n$ is a non-zero hence injective map $\mo_S\ra I_Z\otimes L$, therefore the sheaf $\F$ is a flat family of purely 1-dimensional sheaves on $S$ parametrized by $\mc_{\ls}^n$.  Actually for any $(Z,C)\in\mc_{\ls}^n$, $\F_{(Z,C)}\cong I_{Z/C}\otimes L$ where $I_{Z/C}$ is the ideal sheaf of $Z$ in $C$.  This is because that $\F_{(Z,C)}$ lies in the following exact sequence
\begin{equation}\label{ufrhp}0\ra\mo_S\ra I_Z\otimes L\ra \F_{(Z,C)}\ra 0.
\end{equation} 
One can see that $\det(\F_{(Z,C)})=L=dH$ and $\chi(\F_{(Z,C)})=\chi(L)-n-1$.  If $C$ is integral, then $\F_{(Z,C)}$ is stable.  Let $\chi=\chi(L)-n-1$.  

Denote by $\F_{int}$ ($\F_D$, resp.) the restriction of $\F$ to $S\times \mc_{\ls^{int}}^n$ ($S\times \mc_D^n$, resp.).  We have a classifying map $f:\mc_{\ls^{int}}^n\ra M(d,\chi)$ induced by $\F_{int}$. Denote by $M(d,\chi)_{int}$ ($M(d,\chi)_D$, resp.) the pullback of $\ls^{int}$ ($D$, resp.) via the Hilbert-Chow morphism $\pi:M(d,\chi)\ra\ls$.  Obviously $f(\mc_{\ls^{int}}^n)\subset M(d,\chi)_{int}$, $f(\mc_{D}^n)\subset M(d,\chi)_{D}$ and we have the following commutative diagram
\begin{equation}\label{cmcd}\xymatrix@C=1.2cm{\mc_D^n\ar@{^{(}->}[r]\ar[rrdd]_{\pi_l}\ar@/^2pc/[rrr]^{f}&\mc_{\ls^{int}}^n\ar[r]^{f\quad}\ar[rd]_{\pi_l}&M(d,\chi)_{int}\ar[d]^{\pi}&M(d,\chi)_D\ar@{_{(}->}[l]\ar[ldd]_{\pi}\\&&\ls^{int}&\\ &&D\ar@{^{(}->}[u]&}.\end{equation}
\begin{lemma}\label{ctom}For $n>d^2-3d$, $f$ in (\ref{cmcd}) is a projective bundle.  In particular $\mc_{\ls^{int}}^n$ is smooth and $\mc_{D}^n$ is smooth for a generic $\p^{d-2}\cong D\subset\ls^{int}$. 
\end{lemma}
\begin{proof}The fiber of $f$ over any $\mf\in M(d,\chi)_{int}$ is $\p(\Ext^1(\mf,\mo_S))$.  If $n>d^2-3d$, then for any $\mf\in M(d,\chi)_{int}$, $\deg (\mf\otimes K_S)<0$ and $h^0(\mf\otimes K_S)=0$.  Hence $\dim\Ext^1(\mf,\mo_S)=-\chi(\mf\otimes K_S)=n-\frac{d^2-3d}2$ is a constant for any $\mf\in M(d,\chi)_{int}$.  The smoothness follows from the smoothness of $M(d,\chi)$.  Hence the lemma.
\end{proof}
From now on we let $n>d^2-3d$ and $D$ generic,  hence $\mc_{\ls^{int}}^n$ and $\mc_D^n$ are smooth.  Define $H_{(n)}:=\det ([H-\mo_S]^{[n]})=\int_H \text{ch}_2(\mo_{\Sigma_n})$ to be a line bundle over $S^{[n]}$. \footnote{~In fact $H_{(n)}$ can be defined more concretely:   we have a Hibert-Chow morphism $S^{[n]}\ra S^{(n)}$ with $S^{(n)}$ the $n$-th symmetric power of $S$,  then $H_{(n)}$ is the pullback of the line bundle $(H^{\boxtimes n})^{\mathfrak{S}_n}$ on $S^{(n)}$ where $\mathfrak{S}_n$ is the $n$-th symmetric group.} 

Recall that we have a universal sheaf $\E$ on $M(d,\chi)$ and $\text{ch}^{\alpha}(\E)\in A^*(S\times M(d,\chi))$ defined in (\ref{defc2}).

\begin{lemma}\label{pbea}We have $(f^S)^*\emph{ch}^{\alpha}(\E)=\emph{ch}(\F_{int})\cdot \emph{exp}(\frac 1dp_c^*\varphi_s^*H_{(n)})=:\emph{ch}^{H_{(n)}}(\F)$.
\end{lemma}
\begin{proof}Since $f$ is the classifying map induced by $\F_{int}$, $\exists ~\ma\in A^1(\mc_{\ls^{int}}^n)$ such that $\F_{int}\otimes p_c^*\ma\cong (f^S)^*\E$.  Therefore
\[(f^S)^*\text{ch}^{\alpha}(\E)=\text{ch}(\F_{int})\cdot\text{exp}((f^S)^*\alpha+ p_c^*\ma).\]
By (\ref{defal2}) we have 
$$(f^S)^*\alpha=-\frac 1d(f^S)^*\pi_M^*\left(\int_H\text{ch}_2(\E)\right)=-\frac 1dp_c^*f^*\left(\int_H\text{ch}_2(\E)\right)$$
Hence we only need to show in $A^1(\mc^n_{\ls^{int}})$
\[-f^*\left(\int_H\text{ch}_2(\E)\right)+\ma=\varphi_s^*H_{(n)}.\]

On the other hand, we know that $\text{ch}_0(\F_{int})=0$ and $\text{ch}_1(\F_{int})$ is the class of the support of $\F_{int}$.  Hence $\text{ch}_1(\F_{int})=(\imath^S)^*(\pi_l^S)^*[\mc_{\ls}]=dq^*H+p_c^*\varphi_l^*\xi$.  Therefore for any $\widetilde{\ma}\in A^1(\mc_{\ls^{int}}^n)$,
\[(\text{ch}(\F)\cdot\text{exp}(p_c^*\widetilde{\ma}))_2=\text{ch}_2(\F)+(dq^*H+p_c^*\varphi_l^*\xi)\cdot (p_c^*\widetilde{\ma}),\]
and
\[\int_H(\text{ch}(\F)\cdot\text{exp}(p_c^*\widetilde{\ma}))_2=\int_H\text{ch}_2(\F)+d\widetilde{\ma}\in A^1(\mc_{\ls^{int}}^n),\]
Therefore it is enough to show 
\[\int_H\text{ch}^{H_{(n)}}(\F)_2=\int_H ((f^S)^*\text{ch}^{\alpha}(\E))_2=f^*\left(\int_H(\text{ch}^{\alpha}(\E))_2\right)=f^*c_1(1)=0.\]
By (\ref{ufrh}), it suffices to show
\[\int_H(\text{ch}(\mi_n\otimes q^*(dH))\cdot \text{exp}(\frac 1d p_n^*H_{(n)}))_2=0.\]
Since $\text{ch}(\mi_n)=1-\text{ch}(\mo_{\Sigma_n})$, $\text{ch}_i(\mo_{\Sigma_n})=0$ for $i=0,1$, we have
\begin{eqnarray}&&\int_H(\text{ch}(\mi_n\otimes q^*(dH))\cdot \text{exp}(\frac 1d p_n^*H_{(n)}))_2\nonumber\\
&=&\int_H(\text{exp}(dq^*H+\frac1dp_n^*H_{(n)}))_2-\int_H(\text{ch}(\mo_{\Sigma_n})\cdot \text{exp}(dq^*H+\frac1dp_n^*H_{(n)}))_2\nonumber\\
&=&H_{(n)}-\int_H\text{ch}_2(\mo_{\Sigma_n})=0.\nonumber
\end{eqnarray}
The lemma is proved.
\end{proof}

By Lemma \ref{pbea} and (\ref{ufrh}) we have
\[ (f^S)^*\text{ch}^{\alpha}(\E)=(\imath^S)^*\big((\pi_s^S)^*\text{ch}(\mi_n\otimes q^*(dH))-(\pi_l^S)^*p_l^*\text{exp}(-\xi)\big)\cdot \text{exp}(\frac1dp_c^*\varphi_s^*H_{(n)}).\]
By a direct computation we have 
\begin{eqnarray}&&f^*c_k(j)=\int_H^{j}(f^S)^*(\text{ch}^{\alpha}(\E))_{k+1}\nonumber\\
&=&\int_{H^j}\big\{(\imath^S)^*\big((\pi_s^S)^*\text{ch}(\mi_n\otimes q^*(dH))-(\pi_l^S)^*p_l^*\text{exp}(-\xi)\big)\cdot \text{exp}(\frac1dp_c^*\varphi_s^*H_{(n)})\big\}_{k+1}\nonumber\\
&=&\int_{H^j}\big\{(\imath^S)^*\big(((\pi_s^S)^*\text{ch}(\mi_n\otimes q^*(dH))-(\pi_l^S)^*p_l^*\text{exp}(-\xi))\cdot \text{exp}(\frac1d(\pi_s^S)^*p_n^*H_{(n)})\big)\big\}_{k+1}\nonumber\\
&=&\imath^*\left(\pi_s^*\int_{H^j}(\text{ch}(\mi_n)\cdot\text{exp}(\frac1dp_n^*H_{(n)}+q^*(dH)))_{k+1}-\delta_{j,2}(\text{exp}(\frac1d\pi_s^*p_n^*H_{(n)}-\xi))_{k+1}\right),\nonumber\\
\end{eqnarray}
where $\delta_{j,2}=\begin{cases}1,j=2\\0,\text{othewise}\end{cases}$.
Notice that $\xi$ stands for both the hyperplane class on $\ls$ and its pullback to $S^{[n]}\times \ls$.  Define
\begin{equation}\label{defga}\gamma_k(j):=\pi_s^*\int_{H^j}(\text{ch}(\mi_n)\cdot\text{exp}(\frac1dp_n^*H_{(n)}+q^*(dH)))_{k+1}-\delta_{j,2}(\text{exp}(\frac1d\pi_s^*p_n^*H_{(n)}-\xi))_{k+1}.\end{equation}
Then 
\begin{equation}\label{gatoc}\gamma_k(j)\in A^{i+k-1}(S^{[n]}\times \ls) \text{ and }\imath^*\gamma_k(j)=f^*c_k(j),\forall~k,j.
\end{equation}
Let $\Gamma$ be the subring of $A^*(S^{[n]}\times \ls)$ generated by $\{\gamma_k(j)\}_{k,j}$.  Let $\Gamma^m:=\Gamma\cap A^{\frac{m}2}(S^{[n]}\times\ls)$.  For every $c\in A^{\frac{m}2}(M(d,\chi))=H^m(M(d,\chi),\bq)$, we can find an element $\gamma(c)\in \Gamma^m$ such that $\imath^*(\gamma(c))=f^*c$.  The choice of $\gamma(c)$ may not be unique since $\imath^*$ is not injective.  

Recall we have defined in (\ref{defpx}) that
\[P^{[n]}(\xi):=P_n^{[n]}+P_{n-1}^{[n]}\xi+\cdots+P_1^{[n]}\xi^{n-1}+\xi^{n}\in A^n(S^{[n]}\times \ls).\]
Define a increasing filtration $\widetilde{P}_{\bullet}$ on $\Gamma$ as follows.
\begin{equation}\label{alpf}\widetilde{P}_k\Gamma^m:=\left\{\gamma\in\Gamma^m\left| \gamma\cdot P^{[n]}(\xi)\in \sum_{i\geq 1}\Ker(\xi^{\ttb+k+i-m})\cap \big(\xi^{i-1}\cdot P^{[n]}(\xi)\cdot\Gamma\big)\right.\right\}.\end{equation}
We say a $\gamma\in \Gamma$ is \emph{almost of perversity $k$} if $\gamma\in \widetilde{P}_k\Gamma\setminus \widetilde{P}_{k-1}\Gamma.$

\begin{prop}\label{fmtoh}For every $c\in H^{m}(M(d,\chi),\bq)$, let $\gamma(c)\in \Gamma^m$ such that $\imath^*(\gamma(c))=f^*c$.  Let $m-k\leq d-1$.  Then the following two statements are equivalent.
\begin{itemize}
\item[(1)]$c\in P_kH^m(M(d,\chi),\bq)$;

~~

\item[(2)]$\gamma(c)\in \widetilde{P}_k\Gamma^m$.
\end{itemize}
In particular if $m-k\leq d-2$.  Then the following two statements are equivalent.
\begin{itemize}
\item[(1)]$c$ is of perversity $k$;

~~

\item[(2)]$\gamma(c)$ is almost of perversity $k$.
\end{itemize}

\end{prop}
\begin{proof}Recall that we all denote by $\xi$ the hyperplane class on $\ls$ and its pullbacks to $M(d,\chi)$ and $S^{[n]}\times \ls$.  

Let $c\in H^m(M(d,\chi),\bq)$ and $m-k\leq d-1$.  
Since $\p^{d-2}\cong D\subset\ls^{int}$ is generic, $M(d,\chi)_D$ is projective and smooth.  Moreover $[M(d,\chi)_D]=[\xi^{\ttb-(d-2)}]$ in $A^{*}(M(d,\chi))$.  We have
\begin{eqnarray}\label{cgeq1}c\cdot[M(d,\chi)_D]&\in& \sum_{i\geq 1}\Ker(\xi^{(d-2)+k+i-m})\cap \big(\xi^{i-1}\cdot [M(d,\chi)_D]\cdot(A^*(M(d,\chi))\big)\nonumber\\&\Leftrightarrow& c\in \Ker([M(d,\chi)_D])+ \sum_{i\geq 1}\Ker(\xi^{\ttb+k+i-m})\cap \Ima(\xi^{i-1})
\nonumber\\&\Leftrightarrow& c\in\sum_{i\geq 1}\Ker(\xi^{\ttb+k+i-m})\cap \Ima(\xi^{i-1})\Leftrightarrow c \in P_kH^m(M(d,\chi),\bq).\end{eqnarray}
where the one before the last equivalence is because $m-k\leq d-1$ and hence $\Ker([M(d,\chi)_D])=\Ker(\xi^{\ttb-(d-2)})\subset\Ker(\xi^{\ttb+k+1-m})\cap\Ima(\xi^0)$, and the last equivalence is because of Proposition \ref{desp}.

On the other hand, we have a commutative diagram
\[\xymatrix{\mc_D^n\ar[r]^{f\quad} \ar[dr]_{\pi_l}&M(d,\chi)_D\ar@{^{(}->}[r]^{\jmath}\ar[d]^{\pi}&M(d,\chi)\\ & D&},\] 
where $f$ is a projective bundle.  $f^*:A^*(M(d,\chi)_D)\ra A^*(\mc_D^n)$ is injective.  By Proposition \ref{clofc} $[\mc_D^n]=\xi^{\ttb-(d-2)}\cdot P^{[n]}(\xi)$ in $A^*(S^{[n]}\times \ls)$.  Therefore
\begin{eqnarray}\label{cgeq2}&&c\cdot[M(d,\chi)_D]\in \sum_{i\geq 1}\Ker(\xi^{(d-2)+k+i-m})\cap \big(\xi^{i-1}\cdot [M(d,\chi)_D]\cdot(A^*(M(d,\chi))\big)\nonumber\\&\Leftrightarrow& f^*c \in \sum_{i\geq 1}\Ker(\xi^{(d-2)+k+i-m})\cap \big(\xi^{i-1}\cdot f^*\jmath^*(A^*(M(d,\chi)))\big)
\nonumber\\&\Leftrightarrow& \gamma(c)\cdot[\mc_D^n]\in \sum_{i\geq 1}\Ker(\xi^{(d-2)+k+i-m})\cap \big(\xi^{i-1}\cdot[\mc_D^n]\cdot \Gamma\big)
\nonumber\\&\Leftrightarrow& \gamma(c)\cdot \xi^{\ttb-(d-2)}\cdot P^{[n]}(\xi)\in\sum_{i\geq 1}\Ker(\xi^{(d-2)+k+i-m})\cap \big(\xi^{\ttb-(d-2)+i-1}\cdot P^{[n]}(\xi)\cdot \Gamma\big)\nonumber\\ &\Leftrightarrow& \gamma(c)\cdot P^{[n]}(\xi)\in\sum_{i\geq 1}\Ker(\xi^{\ttb+k+i-m})\cap \big(\xi^{i-1}\cdot P^{[n]}(\xi)\cdot\Gamma\big)+\Ker(\xi^{\ttb-(d-2)})\nonumber\\&\Leftrightarrow& \gamma(c) \in \widetilde{P}_k\Gamma_m,\nonumber\\
\end{eqnarray}
where the last equivalence is because $\gamma(c)\cdot P^{[n]}(\xi)\subset P^{[n]}(\xi)\cdot \Gamma$ and for $m-k\leq d-1$,
$$\Ker(\xi^{\ttb-(d-2)})\cap \big(P^{[n]}(\xi)\cdot\Gamma\big)\subset\Ker(\xi^{\ttb+k+1-m})\cap \big(P^{[n]}(\xi)\cdot\Gamma\big).$$
The proposition is proved.
\end{proof}

Since for any $\chi\in\bz$, one can find a $\chi'>d^2-3d$ such that $\chi'\equiv\chi~(d)$.  Together with Proposition \ref{fmtoh} and Remark \ref{chchi}, Theorem \ref{main3} below implies Theorem \ref{main1}.
\begin{thm}\label{main3}Let $d\geq 3$.  For any $n\geq 2$, on $S^{[n]}\times \ls$ we have
\begin{itemize}\item[(1)]$\gamma_1(2)$ is almost of perversity 1;

~~

\item[(2)]$\gamma_2(1)$ is almost of perversity 2;

~~

\item[(3)]$\gamma_3(0)$ is almost of perversity 3.

\end{itemize}
\end{thm}

Now Theorem \ref{main3} only concerns computations in $A^*(S^{[n]}\times\ls)$.  Our main strategy to prove Theorem \ref{main3} is to do induction on $n$. 

By K\"unneth formula we have 
\begin{equation}\label{Kde}H^k(S\times S^{[n]},\bz)\cong \bigoplus_{i+j=k}H^i(S,\bz)\otimes H^j(S^{[n]},\bz)\cong \bigoplus_{i+j=k}A^i(S)\otimes A^j(S^{[n]}).\end{equation}
Recall $A^*(S)=\bz \mathbf{1}\oplus\bz H\oplus \bz H^2$ with $H$ the hyperplane class.  Write
\begin{equation}\label{dese1}c_i(\mo_{\Sigma_n})=s_i(\mi_n)=\theta_i^2(n)\otimes \mathbf{1}+\theta_i^{1}(n)\otimes H+\theta_i^0(n)\otimes H^2,
\end{equation}
with $\theta_i^2(n)\in A^i(S^{[n]})$, $\theta_i^{1}(n)\in A^{i-1}(S^{[n]})$ and $\theta_i^0(n)\in A^{i-2}(S^{[n]})$\footnote{Here we use different superscripts from \cite{Yuan11}, in order to be consistent with \cite{KPS}}.  Since $\Sigma_n$ is of codimension 2 in $S\times S^{[n]}$, we have $\theta_1^{i}(n)=0$ for $i=0,1,2$, $\theta_0^{i}(n)=0$ for $i=0,1$ and $\theta_0^2(n)=[S^{[n]}]=\mathbf{1}$.
Also easy to see that $H_{(n)}=-\theta_2^1(n)$ and $\theta_2^0(n)=-n\cdot\mathbf{1}$.  

Denote also by $\theta_i^j(n)$ its pullback to $S^{[n]}\times\ls$.  By (\ref{defga}) and a direct computation we have 
\begin{eqnarray}\label{exfga}&&\gamma_0(0)=0,\quad\gamma_0(1)=d\cdot\mathbf{1},\quad\gamma_0(2)=\xi;\nonumber\\
&&\gamma_1(0)=\frac{d^2}2-n,\quad\gamma_1(1)=0,\quad \gamma_1(2)=\theta_2^2(n)-\frac1d\theta_2^1(n)\xi-\frac{\xi^2}2;\nonumber\\
&&\gamma_2(0)=\frac12\left(\left(d+\frac{2n}d\right)\theta_2^1(n)-\theta_3^0(n)\right), \gamma_2(1)=-\frac12\left(\frac1d\theta_2^1(n)^2+\theta_3^1(n)-2d\theta_2^2(n)\right),\nonumber\\ &&\gamma_3(0)=\frac16\left(\theta_4^0(n)+\frac3d\theta_2^1(n)\theta_3^0(n)-(\frac{3n}{d^2}+2)\theta_2^1(n)^2\right)+d\gamma_2(1)+\left(\frac n6-\frac{d^2}2\right)\theta_2^2(n).\nonumber\\
\end{eqnarray}


\subsection{The incidence variety for Hilbert schemes of points.} In this subsection, we review the set-up in \cite{EGL} (which in fact applies to $S$ being any projective complex surface).  

Define $\p(\mi_n):=Proj(Sym^*(\mi_n))$.  Then $\p(\mi_n)$ is isomorphic to the incidence variety $S^{[n,n+1]}$ parametrizing all pairs $(Z,Z')\in S^{[n]}\times S^{[n+1]}$ satisfying $Z\subset Z'$.  We have two classifying morphisms $\psi_n:\p(\mi_n)\ra S^{[n+1]},(Z,Z')\mapsto Z'$ and $\phi_n:\p(\mi_n)\ra S^{[n]},(Z,Z')\mapsto Z$.

It is known that $\p(\mi_n)$ is irreducible and smooth (see e.g. \cite{Che},\cite{Tik},\cite{ES}).

We have the projection $\sigma_n=(\phi_n,\rho_n):\p(\mi_n)\ra S^{[n]}\times S$ with tautological quotient line bundle $\mo_{\sigma_n}(1)=:\cl_n$.  Then $(\sigma_n)_{*}\cl_n\cong \mi_n$.  
We have the following commutative diagram
\begin{equation}\label{pcd}\xymatrix@C=1.2cm{\p(\mi_n)\cong S^{[n,n+1]}\ar[d]_{\psi_n\qquad}\ar[r]^{\qquad\sigma_n}\ar[rd]^{\phi_n}\ar@/^2pc/[rr]^{\rho_n}&S^{[n]}\times S\ar[d]^{p_n}\ar[r]^{\quad q_n}&S\\ S^{[n+1]}& S^{[n]}&}
\end{equation}

We use the same letter $\cl_n$ as the divisor class $c_1(\cl_n)$.  By Lemma 1.1 in \cite{EGL} we have 
\begin{equation}\label{psl}(\sigma_n)_*((-\cl_n)^i)=c_i(\mo_{\Sigma_n}).\end{equation}

Recall that for any $\alpha\in K(S)$ we have defined $\alpha^{[n]}\in K(S^{[n]})$ in (\ref{ghs}). 
\begin{lemma}[Lemma 2.1 in \cite{EGL}]\label{ind}The following relation holds in $K(\p(\mi_n))$
\[\psi_n^!\alpha^{[n+1]}=\phi_n^!\alpha^{[n]}+\cl_n\cdot\rho_n^!\alpha\]
for any $\alpha\in K(S)$.
\end{lemma} 

By Lemma \ref{ind} we have
\begin{equation}c_t(\phi_n^!\alpha^{[n]})c_t(\cl_n\cdot\rho_n^!\alpha)=c_t(\psi_n^!\alpha^{[n+1]}).\end{equation}
In particular, let $\alpha=L$ then we have
\begin{equation}\label{pnrn}\psi_n^*(P^{[n+1]}_{n+1})=\phi_n^*(P^{[n]}_n)\cdot (\cl_n+\rho_n^*L),
\end{equation}
\begin{equation}\label{pnro}\psi_n^*(P^{[n+1]}_i)=\phi_n^*(P^{[n]}_i)+\phi_n^*(P^{[n]}_{i-1})\cdot (\cl_n+\rho_n^*L),\text{ for }i\leq n,
\end{equation}
and 
\begin{equation}\label{pnr}\psi_n^*(P^{[n+1]}(\xi))=\phi_n^*(P^{[n]}(\xi))\cdot(1+(\cl_n+\rho_n^*L)\xi).
\end{equation}

\subsection{Some recursion relations.}
We have a commutative diagram
\begin{equation}\label{pcd2}\xymatrix@C=2cm{\p(\mi_n)&\p(\mi_n)\times S\ar[l]_{p}\ar[d]_{\psi_n^S:=\psi_n\times id_S}\ar[r]^{\sigma_n^S:=\sigma_n\times id_S}\ar[rd]^{\phi_n^S:=\phi_n\times id_S}\ar@/^2pc/[rr]^{\rho_n^S:=\rho_n\times id_S}&S^{[n]}\times S\times S\ar[d]^{p_n^S:=p_n\times id_S}\ar[r]^{\quad q_n^S:=q_n\times id_S}&S\times S\\ &S^{[n+1]}\times S& S^{[n]}\times S&}
\end{equation}

Denote by $\Delta\subset S\times S$ the diagonal.  By \cite{EGL} we have
\begin{equation}\label{exid1}0\ra (\psi^S_n)^*\mi_{n+1}\ra(\phi_n^S)^*\mi_n\ra p^*\cl_n\otimes (\rho_n^S)^*\mo_{\Delta}\ra 0.\end{equation}
\begin{equation}\label{exid}0\ra p^*\cl_n\otimes (\rho_n^S)^*\mo_{\Delta}\ra (\psi^S_n)^*\mo_{\Sigma_{n+1}}\ra(\phi_n^S)^*\mo_{\Sigma_n} \ra 0.\end{equation}
We also have the following lemma proved in \cite{Yuan11}.

\begin{lemma}[Lemma 2.2 in \cite{Yuan11}]\label{flat}(1) The map $\rho_n^S$ is flat.

(2) $(\psi^S_n)^*\mi_{n+1}\cong (\psi^S_n)^!\mi_{n+1}$ and $(\phi_n^S)^*\mi_n\cong (\phi_n^S)^!\mi_n$.
\end{lemma}

By (\ref{exid}) and Lemma \ref{flat} we have
\begin{equation}\label{indin}c_t((\psi_n^S)^!\mo_{\Sigma_{n+1}})=c_t((\phi_n^S)^!\mo_{\Sigma_n})c_t(\cl_n\cdot(\rho_n^S)^!\mo_{\Delta}).\end{equation}

Recall we have 
\begin{equation}\label{dese}c_i(\mo_{\Sigma_n})=\theta_i^2(n)\otimes \mathbf{1}+\theta_i^{1}(n)\otimes H+\theta_i^0(n)\otimes H^2.
\end{equation}

We make the convention that $\theta_i^j(n)=0$ and $\cl_n^i=0$ for $i<0$.  It is well-known that $A^*(S^{[n]})\cong H^*(S^{[n]},\bq)$ is generated by $\{\theta_i^j(n)\}_{i,j}$ as a ring (see e.g. \cite{mark}).

%
Define
\begin{eqnarray}\label{defabc}A_i(n)&:=&\sum_{k= 0}^{i-2}\phi_n^*\theta_{i-k-2}^2(n)(-\cl_n)^{k}(k+1)\nonumber\\
B_i(n)&:=&3\sum_{k=0}^{i-3}\phi_n^*\theta_{i-k-3}^2(n)(-\cl_n)^{k}\binom{k+2}{2}\nonumber\\ &&+\sum_{k=2}^{i-2}\phi_n^*\theta_{i-k-2}^1(n)(-\cl_n)^{k}(k+1)\nonumber\\
C_i(n)&:=&6\sum_{k=0}^{i-4}\phi_n^*\theta_{i-k-4}^2(n)(-\cl_n)^{k}\binom{k+3}{3}\nonumber\\ &&+3\sum_{k=2}^{i-3}\phi_n^*\theta_{i-k-3}^1(n)(-\cl_n)^{k}\binom{k+2}{2}\nonumber\\ &&+\sum_{k=2}^{i-2}\phi_n^*\theta_{i-k-2}^0(n)(-\cl_n)^{k}(k+1).\end{eqnarray}

We have the following proposition.
\begin{prop}\label{rrth}For any $n\geq 1$, we have
\begin{eqnarray}\psi_n^*\theta_i^2(n+1)&=&\phi_n^*\theta_i^2(n)-A_i(n)\rho_n^*H^2,\nonumber\\
\psi_n^*\theta_i^1(n+1)&=&\phi_n^*\theta_i^1(n)-A_i(n) \rho_n^*H-B_i(n)\rho_n^*H^2,\nonumber\\
\psi_n^*\theta_i^0(n+1)&=&\phi_n^*\theta_i^0(n)-A_i(n)-B_i(n) \rho_n^*H-C_i(n)\rho_n^*H^2.\nonumber
\end{eqnarray}
\end{prop}
\begin{proof}The proof is the same as that of Proposition 3.1 in \cite{Yuan11}. (Be careful we define $\theta_i^j(n)$ differently from \cite{Yuan11}.)  

By Beilinson spectral sequence we can compute the Chern polynomial $c_t(\mo_{\Delta})$ as follows.
\begin{equation}\label{cpdel}c_t(\mo_{\Delta})=\frac{1\cdot(1-p^*_1Ht-p_2^*Ht)}{(1-\gamma_1t-p_2^*Ht)(1-\gamma_2t-p_2^*Ht)},\end{equation}
where $p_i:S\times S\ra S$ is the $i$-th projection, $\gamma_1+\gamma_2=p_1^*H$ and $\gamma_1\cdot\gamma_2=p_1^*H^2$.

The proposition follows from  (\ref{indin}) and a direct computation.
\end{proof}

Define $\Omega:=\psi^*_n(H^*(S^{[n+1]},\bq))=\psi^*_n(A^*(S^{[n+1]}))\subset H^*(\p(\mi_n),\bq)$.  Let 
$$\Omega^{\bot}:=\{x\in H^*(\p(\mi_n),\bq)\big|\forall ~y\in\Omega, \int_{\p(\mi_n)}x\cdot y=0\}.$$
Since $\psi_n^*$ is surjective and generically of degree $n+1$, $\Omega\cap\Omega^{\bot}=\{0\}.$
\begin{lemma}\label{omep}For any $x\in H^*(\p(\mi_n),\bq)$, if $(\phi_n)_*(x\cdot \rho_n^*H^{j}\cdot\cl_n^{i})=0$ for all $i,j\in\bz_{\geq0}$, then $x\in \Omega^{\bot}$.  In particular for any $x\in \Omega$, we have 
\[x=0\Leftrightarrow (\phi_n)_*(x\cdot \rho_n^*H^{j}\cdot\cl_n^{i})=0\text{ for all }i,j\in\bz_{\geq0}.\]
\end{lemma}
\begin{proof}By definition, $x\in \Omega^{\bot}$ iff for any polynomial $F$ in $\theta_{i}^j(n+1),i,j\in\bz_{\geq 0}$, we have $\int_{\p(\mi_n)}x\cdot\psi_n^*F=0$.  We also have
\[\int_{\p(\mi_n)}x\cdot\psi_n^*F=\int_{S^{[n]}}(\phi_n)_*(x\cdot\psi_n^*F).\] 
On the other hand, by Proposition \ref{rrth} $\psi_n^*F$ equals to a polynomial $G$ with variables $\phi_n^*\theta_i^j(n),i,j\in\bz_{\geq 0} $, $\rho_n^*H$ and $\cl_n$.  Therefore there are finitely many polynomials $G_{s,t}$ in $\theta_i^j(n),i,j\in\bz_{\geq 0} $ such that 
\[G=\sum_{s,t\geq 0}\phi_n^*(G_{s,t})\cdot(\rho_n^*H)^s\cdot\cl_n^{t}.\] 
Therefore
\begin{eqnarray}\int_{\p(\mi_n)}x\cdot\psi_n^*F&=&\int_{S^{[n]}}(\phi_n)_*(x\cdot\psi_n^*F)
=\int_{S^{[n]}}(\phi_n)_*(x\cdot G)\nonumber\\ &=&\sum_{s,t\geq 0}\int_{S^{[n]}}G_{s,t}\cdot(\phi_n)_*(x\cdot(\rho_n^*H)^s\cdot\cl_n^t).
\end{eqnarray}
The lemma follows.
\end{proof}

\section{Proof of the main theorem.}
We prove Theorem \ref{main3} in this section.  In this section and also appendix we will use the same letter for a class and its pullback if there is no confusion.  

Recall that
\[P^{[n]}(\xi):=P_n^{[n]}+P_{n-1}^{[n]}\xi+\cdots+P_1^{[n]}\xi^{n-1}+\xi^{n}\in A^n(S^{[n]}\times \ls),\]
with $P_i^{[n]}$ (the pullback to $S^{[n]}\times\ls$ of) the $i$-th Chern class of the rank $n$ vector bundle $[dH]^{[n]}$ on $S^{[n]}$ and
\begin{equation}\label{alpf1}\widetilde{P}_k\Gamma^m:=\left\{\gamma\in\Gamma^m\left| \gamma\cdot P^{[n]}(\xi)\in \sum_{i\geq 1}\Ker(\xi^{\ttb+k+i-m})\cap \big(\xi^{i-1}\cdot P^{[n]}(\xi)\cdot\Gamma\big)\right.\right\}.\end{equation}
We say a $\gamma\in \Gamma$ is \emph{almost of perversity $k$} if $\gamma\in \widetilde{P}_k\Gamma\setminus \widetilde{P}_{k-1}\Gamma$. Notice that $P_i^{[n]}$, $\Gamma$, $\widetilde{P}_{\bullet}$ all depend on $d$ and from now on we always assume $d\geq 3$.

Let $A_i(n),B_i(n),C_i(n)$ be as in (\ref{defabc}).
\subsection{The computation for $\gamma_1(2)$.} 
By (\ref{exfga}) we have 
$$\gamma_1(2)=\theta_2^2(n)-\frac1d\theta_2^1(n)\xi-\frac{\xi^2}2\in H^4(S^{[n]}\times\ls).$$ 
 
\begin{lemma}\label{pnt2}In $H^*(S^{[n]},\bq)$ we have for all $i\geq 1$
\begin{enumerate}
\item $\theta_i^2(n)P_n^{[n]}=0$;

~~

\item $\theta_i^2(n)P_{n-1}^{[n]}-\frac1d\theta_i^1(n)P_n^{[n]}=0$;

~~

\item $\theta_i^2(n)P_{n-2}^{[n]}-\frac1d\theta_i^1(n)P_{n-1}^{[n]}+\frac1{d^2}\theta_i^0(n)P_n^{[n]}=0$.
\end{enumerate}
Moreover in $H^*(\p(\mi_n),\bq)$ we have
\begin{enumerate}
\item[(4)] $H^2\cl_n P_n^{[n]}=0$;

~~

\item[(5)] $H^2\cl_n P_{n-1}^{[n]}-\frac1dH\cl_n P_n^{[n]}=0$;

~~~

\item[(6)] $H^2\cl_n P_{n-2}^{[n]}-\frac1dH\cl_n P_{n-1}^{[n]}+\frac1{d^2} \cl_n P_n^{[n]}=0$.
\end{enumerate}
\end{lemma}
\begin{proof}We do induction on $n$.  When $n=1$, we only need to check for $i=2,3$.  We have $S^{[1]}=S$ and the universal family $\Sigma_1\subset S\times S$ is the diagonal $\Delta$.  Hence by (\ref{cpdel}) $\theta_2^0(n)=-\mathbf{1},\theta_2^1(1)=-H,\theta_2^2(1)=-H^2,\theta_3^0(1)=-3H,\theta_3^1(1)=-3H^2$, $\theta_3^2(n)=0$, $P_{1}^{[1]}=dH$, $P_0^{[1]}=\mathbf{1}$ and $P_k^{[1]}=0$ for $k<0$.  It is easy to check (1)-(3) hold for $n=1$.  

Assume (1)-(3) hold for $n$.  We need to show that they hold for $n+1$.  

It is enough to show that the pullback of (1)-(3) from $S^{[n+1]}$ to $\p(\mi_n)$ hold.  By Proposition \ref{rrth} and (\ref{pnrn}) we have on $\p(\mi_n)$
\begin{eqnarray}\label{thpn}\theta^2_i(n+1)P_{n+1}^{[n+1]}&=&(\theta^2_i(n)-H^2A_i(n))P_n^{[n]}(\cl_n+dH)\nonumber\\
&=&\theta_i^2(n)P_n^{[n]}(\cl_n+dH)-A_i(n)H^2\cl_nP_n^{[n]}\nonumber\\
&=&-A_i(n)H^2\cl_nP_n^{[n]}.
\end{eqnarray} 

By (\ref{psl}) we have $(\phi_n)_*(H^j(-\cl_n)^i)=\theta_i^{j}(n)$ for $j=0,1,2$.  Since $\theta_i^2(n)P_n^{[n]}=0$ for all $i\geq 1$ by induction assumption, we have 
$$(\phi_n)_*\big((H^2\cl_nP_n^{[n]})\cdot H^i\cdot\cl_n^j\big)=0\text{ for all }i,j\in\bz_{\geq 0}$$ and hence $H^2\cl_nP_n^{[n]}\in \Omega^{\bot}$ by Lemma \ref{omep}.

On the other hand, $A_2(n)=\mathbf{1}$ and hence by (\ref{thpn}) 
$$\theta_2^2(n+1)P_{n+1}^{[n+1]}=-H^2\cl_nP_n^{[n]}.$$ 
Therefore $H^2\cl_nP_n^{[n]}\in\Omega\cap\Omega^{\bot}$ and hence $H^2\cl_nP_n^{[n]}=0$ and also $\theta^2_i(n+1)P_{n+1}^{[n+1]}=0$ for all $i\geq 1$.  We have proved (1) and (4).  

We also have on $\p(\mi_n)$
\begin{eqnarray}\label{thpn1}&&\theta^2_i(n+1)P_{n}^{[n+1]}-\frac1d\theta_i^1(n+1)P_{n+1}^{[n+1]}\nonumber\\&=&(\theta^2_i(n)-H^2A_i(n))(P_n^{[n]}+P_{n-1}^{[n]}(\cl_n+dH)\nonumber\\ &&-\frac1d(\theta_i^1(n)-HA_i(n)-H^2B_i(n))P_n^{[n]}(\cl_n+dH)\nonumber\\
&=&\theta_i^2(n)P_n^{[n]}+(\theta_i^2(n)P_{n-1}^{[n]}-\frac1d\theta_i^1(n)P_n^{[n]})(\cl_n+dH)\nonumber\\ &&-A_i(n)(H^2\cl_nP_{n-1}^{[n]}-\frac1dH\cl_nP_n^{[n]})\nonumber\\
&=&-A_i(n)(H^2\cl_nP_{n-1}^{[n]}-\frac1dH\cl_nP_n^{[n]}),
\end{eqnarray} 
and
\begin{eqnarray}\label{thpn2}&&\theta^2_i(n+1)P_{n-1}^{[n+1]}-\frac1d\theta_i^1(n+1)P_{n}^{[n+1]}+\frac1{d^2}P_{n+1}^{[n+1]}\nonumber\\&=&(\theta^2_i(n)-H^2A_i(n))(P_{n-1}^{[n]}+P_{n-2}^{[n]}(\cl_n+dH)\nonumber\\ &&-\frac1d(\theta_i^1(n)-HA_i(n)-H^2B_i(n))(P_n^{[n]}+P_{n-1}^{[n]}(\cl_n+dH))\nonumber\\
&&+\frac1{d^2}(\theta_i^0-A_i(n)-HB_i(n)-H^2C_i(n))P_n^{[n]}(\cl_n+dH)\nonumber\\
&=&(\theta_i^2(n)P_{n-1}^{[n]}-\frac1d\theta_i^1(n)P_n^{[n]})+(\theta_i^2(n)P_{n-2}^{[n]}-\frac1d\theta_i^1(n)P_{n-1}^{[n]}+\frac1{d^2}\theta_i^0(n)P_n^{[n]})(\cl_n+dH)\nonumber\\ 
&&-A_i(n)(H^2\cl_nP_{n-2}^{[n]}-\frac1dH\cl_nP_{n-1}^{[n]}+\frac1{d^2}\cl_nP_n^{[n]})-B_i(n)(H^2\cl_nP_{n-1}^{[n]}-\frac1dH\cl_nP_n^{[n]})\nonumber\\
\end{eqnarray} 
One can prove (2) (3) (5) and (6) analogously.  We leave the rest of the proof to the reader.
\end{proof}

\begin{proof}[Proof of Theorem \ref{main3} (1)]By Lemma \ref{pnt2} on $S^{[n]}\times\ls$ we have
\begin{equation}\label{exg12}\gamma_1(2)P^{[n]}(\xi)+\left(\frac12-\frac{n}{d^2}\right)\xi^2P^{[n]}(\xi)=\xi^3\big(\theta_2^2(n)P_{n-3}^{[n]}-\frac1d\theta_2^1(n)P_{n-2}^{[n]}-\frac{n}{d^2}P_{n-1}^{[n]}\big)+\xi^4\cdot \mh,\end{equation}
for some $\mh\in A^*(S^{[n]}\times\ls)$.  Therefore
\[\gamma_1(2)P^{[n]}(\xi)+\left(\frac12-\frac{n}{d^2}\right)\xi^2P^{[n]}(\xi)\in \Ker(\xi^{\ttb+1-3})\]
and hence $\gamma_1(2)\in \widetilde{P}_1\Gamma^4$. In order to show $\gamma_1(2)\not\in \widetilde{P}_0\Gamma^4$, it suffices to show that $\nexists \gamma\in \Gamma^{2},$ such that 
\begin{equation}\label{apzero}\xi\gamma P^{[n]}(\xi)-\xi^3\big(\theta_2^2(n)P_{n-3}^{[n]}-\frac1d\theta_2^1(n)P_{n-2}^{[n]}-\frac{n}{d^2}P_{n-1}^{[n]}\big)\in \xi^4 \cdot A^*(S^{[n]}\times \ls).\end{equation}

Assume we can find $\gamma\in\Gamma^2$ satisfying (\ref{apzero}), we write $\gamma=a\gamma_0(2)+b\gamma_2(0)=a\xi+b\gamma_2(0)$ with $a,b\in\bq$.  Then we have
\begin{equation}\begin{cases}b\gamma_2(0)P_n^{[n]}=0\\ b\gamma_2(0)P_{n-1}^{[n]}+aP_n^{[n]}=0\\ b\gamma_2(0)P_{n-2}^{[n]}+aP_{n-1}^{[n]}= \big(\theta_2^2(n)P_{n-3}^{[n]}-\frac1d\theta_2^1(n)P_{n-2}^{[n]}-\frac{n}{d^2}P_{n-1}^{[n]}\big)\end{cases}\end{equation}
Recall that $$\gamma_2(0)=\frac12\left(\left(d+\frac{2n}d\right)\theta_2^1(n)-\theta_3^0(n)\right).$$
By Lemma \ref{inp13} (1) (2) we have 
\begin{equation}\label{nog2}\gamma_2(0)P_n^{[n]}\neq 0,\text{ for }d\neq 1,2.\end{equation}  
Hence $a=b=0$ and $\big(\theta_2^2(n)P_{n-3}^{[n]}-\frac1d\theta_2^1(n)P_{n-2}^{[n]}-\frac{n}{d^2}P_{n-1}^{[n]}\big)=0$.
Therefore it suffices to show 
\begin{equation}\label{nezan}a_n:=\theta_2^2(n)P_{n-3}^{[n]}-\frac1d\theta_2^1(n)P_{n-2}^{[n]}-\frac{n}{d^2}P_{n-1}^{[n]}\neq0.\end{equation}
We may assume $a_{n+1}=0$, then 
\begin{eqnarray}0&=&(\phi_n)_*(\psi_n^*a_{n+1}H)\nonumber\\
&=&(\phi_n)_*\big((\theta_2^2(n)-H^2)(P_{n-2}^{[n]}+P_{n-3}^{[n]}(\cl_n+dH))H\big)\nonumber\\
&&-\frac1d(\phi_n)_*\big((\theta_2^1(n)-H)(P_{n-1}^{[n]}+P_{n-2}^{[n]}(\cl_n+dH))H\big)\nonumber\\
&&-\frac{n+1}{d^2}(\phi_n)_*\big(P_{n}^{[n]}+P_{n-1}^{[n]}(\cl_n+dH)H\big)\nonumber\\
&=&d(\theta_2^2(n)P_{n-3}^{[n]}-\frac1d\theta_2^1(n)P_{n-2}^{[n]}-\frac{n}{d^2}P_{n-1}^{[n]})=da_n.
\end{eqnarray} 
Therefore $a_{n+1}=0\Rightarrow a_n=0$, and hence we must have $a_1=0$ which contradicts with 
$$a_1=-\frac{1}{d^2}P_0^{[1]}=-\frac1{d^2}\cdot\mathbf{1}\neq 0.$$

Therefore $\gamma_1(2)$ is almost of perversity 1.
\end{proof}
\begin{lemma}\label{inp13}On $S^{[n]}$ we have 
\begin{enumerate}
\item[(1)]$\int_{S^{[n]}}\theta_2^1(n)^nP_n^{[n]}=(-1)^nd^n$

~~

\item[(2)]$\int_{S^{[n]}}\left(\theta_3^0(n)-\left(\frac{2(n-1)}d+3\right)\theta_2^1(n)\right)\theta_2^1(n)^{n-1}P_n^{[n]}=0$
\end{enumerate}
\end{lemma}
\begin{proof}We do induction on $n$.  

When $n=1$, $S^{[1]}=S$, $\theta_2^1(1)=-H,\theta_3^0(1)=-3H$, $P_{1}^{[1]}=dH$ and $P_0^{[1]}=\mathbf{1}$.  Hence the lemma holds for this case.

By (\ref{defabc}) we have $A_2(n)=\mathbf{1}$, $B_2(n)=C_2(n)=0$, $A_3(n)=-2\cl_n$, $B_3(n)=3\cdot\mathbf{1}$ and $C_3(n)=0$.  Since $\psi_n^*:\p(\mi_n)\ra S^{[n+1]}$ is surjective and generically of degree $n+1$, we have
\begin{eqnarray}&&(n+1)\int_{S^{[n+1]}}\theta_2^1(n+1)^{n+1}P_{n+1}^{[n+1]}=\int_{\p(\mi_n)}\theta_2^1(n+1)^{n+1}P_{n+1}^{[n+1]}\nonumber\\
&=&\int_{\p(\mi_n)}(\theta_2^1(n)-H)^{n+1}P_n^{[n]}(\cl_n+dH)\nonumber\\
&=&\int_{\p(\mi_n)}\left(\theta_2^1(n)^{n+1}-(n+1)\theta_2^1(n)^nH+\binom{n+1}{2}\theta_2^1(n)^{n-1}H^2\right)P_n^{[n]}(\cl_n+dH)\nonumber\\
&=&\int_{S^{[n]}}\left(-\theta_2^1(n)^{n+1}\theta_1^0(n)+(n+1)\theta_2^1(n)^n(\theta_1^1(n)-d)-\binom{n+1}{2}\theta_2^1(n)^{n-1}\theta_1^2(n)\right)P_n^{[n]}\nonumber\\
&=&-(n+1)d\int_{S^{[n]}}\theta_2^1(n)^{n}P_n^{[n]},\nonumber
\end{eqnarray}
where the last equality is because $(\phi_n)_*(-\cl_nH^j)=\theta_1^j(n)=0$ for $j=0,1,2.$.  By induction assumption 
$\int_{S^{[n]}}\theta_2^1(n)^{n}P_n^{[n]}=(-1)^nd^n$, we have $\int_{S^{[n+1]}}\theta_2^1(n+1)^{n+1}P_{n+1}^{[n+1]}=(-1)^{n+1}d^{n+1}$ and (1) is proved. 

Notice that $\theta_i^2(n)P_n^{[n]}=0$ for all $i\geq 1$ and $H^2\cl_nP_n^{[n]}=0$ by Lemma \ref{pnt2}.  Analogously, we have
\begin{eqnarray}&&(n+1)\int_{S^{[n+1]}}\left(\theta_3^0(n+1)-\left(\frac{2n}d+3\right)\theta_2^1(n+1)\right)\theta_2^1(n+1)^{n}P_{n+1}^{[n+1]}
\nonumber\\&=&\int_{\p(\mi_n)}\left(\theta_3^0(n)+2\cl_n-3H-\left(\frac{2n}d+3\right)(\theta_2^1(n)-H)\right)(\theta_2^1(n)-H)^{n}P_{n}^{[n]}(\cl_n+dH)\nonumber\\
&=&\int_{\p(\mi_n)}(\theta_2^1(n)^n-nH\theta_2^1(n)^{n-1})\left(\theta_3^0(n)+2\cl_n-\left(\frac{2n}d+3\right)\theta_2^1(n)+\frac{2n}dH\right)P_n^{[n]}(\cl_n+dH)\nonumber\\
&=&\int_{\p(\mi_n)}2\cl_n^2\big(\theta_2^1(n)^{n-1}(-H)+\theta_2^1(n)^n\big)P_n^{[n]}+\int_{S^{[n]}}(2+2n)\theta_2^1(n)^nP_n^{[n]}\nonumber\\
&&-d\int_{S^{[n]}}\left(\theta_3^0(n)-\left(\frac{2(n-1)}d+3\right)\theta_2^1(n)\right)\theta_2^1(n)^{n-1}P_n^{[n]}\nonumber\\
&=&\int_{S^{[n]}}2\big(-\theta_2^1(n)^{n}+\theta_2^0(n)\theta_2^1(n)^n\big)P_n^{[n]}+\int_{S^{[n]}}(2+2n)\theta_2^1(n)^nP_n^{[n]}\nonumber\\
&&-d\int_{S^{[n]}}\left(\theta_3^0(n)-\left(\frac{2(n-1)}d+3\right)\theta_2^1(n)\right)\theta_2^1(n)^{n-1}P_n^{[n]}\nonumber\\
&=&-d\int_{S^{[n]}}\left(\theta_3^0(n)-\left(\frac{2(n-1)}d+3\right)\theta_2^1(n)\right)\theta_2^1(n)^{n-1}P_n^{[n]},\nonumber
\end{eqnarray}
where the last equality is because $\theta_2^0(n)=-n\cdot\mathbf{1}$.  Hence (2) is proved.
\end{proof}

\subsection{The computation for $\gamma_2(1)$.}
By (\ref{exfga}) we have
\[\gamma_2(1)=-\frac12\left(\frac1d\theta_2^1(n)^2+\theta_3^1(n)-2d\theta_2^2(n)\right).\]
\begin{lemma}\label{pnt31}In $H^*(S^{[n]},\bq)$ we have for all $i,j\geq 2$
\begin{itemize}
\item[(1)]$(\frac1d\theta_j^1(n)\theta_i^1(n)+\theta_{i+j-1}^1(n))P_n^{[n]}=0.$
\end{itemize}
Moreover in $H^*(\p(\mi_n),\bq)$, we have for all $i,j\geq 2$
\begin{itemize}
\item[(2)]$(-\frac1d\theta_i^1(n)H\cl-\theta_i^1(n)H^2+(-\cl)^iH)P_n^{[n]}=0;$

~~

\item[(3)]$(\frac1d\theta_i^1(n)H(-\cl)^j+(-\cl)^{i+j-1}H)P_n^{[n]}=0.$
\end{itemize}
\end{lemma}
\begin{proof}Since $H^2\cl_nP_n^{[n]}=0$ by Lemma \ref{pnt2} (4), (2) implies (3).  We only need to prove (1) and (2).  We do the induction on $n$.  By the dimension reason, (1) is true for $n=1$.  We assume (1) hold for $n$, then it is enough to show that the pullback of $(\frac1d\theta_j^1(n+1)\theta_i^1(n+1)+\theta_{i+j-1}^1(n+1))P_{n+1}^{[n+1]}$ to $\p(\mi_n)$ is zero.  On $\p(\mi_n)$ we have
\begin{eqnarray}\label{g21tp}&&(\frac1d\theta_j^1(n+1)\theta_i^1(n+1)+\theta_{i+j-1}^1(n+1))P_{n+1}^{[n+1]}\nonumber\\
&=&\frac1d(\theta_j^1(n)-A_j(n)H-B_j(n)H^2)(\theta_i^1(n)-A_i(n)H-B_i(n)H^2)P_{n}^{[n]}(\cl_n+dH)\nonumber\\
&&+\big(\theta_{i+j-1}^1(n)-A_{i+j-1}(n)H-B_{i+j-1}(n)H^2\big)P_{n}^{[n]}(\cl_n+dH)\nonumber\\
&=&\nonumber(\frac1d\theta_j^1(n)\theta_i^1(n)+\theta_{i+j-1}^1(n))P_n^{[n]}(\cl_n+dH)\nonumber\\
&&-(\frac1d\theta_j^1(n)A_{i}(n)H+\frac1d\theta_i^1(n)A_{j}(n)H+A_{i+j-1}(n)H)P_n^{[n]}(\cl_n+dH)\nonumber\\
&=&-(i-1)(\frac1d\theta_j^1(n)(-\cl)^{i-2}H+(-\cl)^{i+j-3})P_n^{[n]}(\cl_n+dH)\nonumber\\
&&-(j-1)(\frac1d\theta_i^1(n)(-\cl)^{j-2}H+(-\cl)^{i+j-3})P_n^{[n]}(\cl_n+dH)\nonumber\\
&=&-(i-1)(\frac1d\theta_j^1(-\cl)^{i-1}H+\theta_j^1(n)(-\cl)^{i-2}H^2+(-\cl)^{i+j-2})P_n^{[n]}\nonumber\\
&&-(j-1)(\frac1d\theta_i^1(-\cl)^{j-1}H+\theta_i^1(n)(-\cl)^{j-2}H^2+(-\cl)^{i+j-2})P_n^{[n]}\nonumber\\\end{eqnarray}
where the one before the last equation is because of the induction assumption and the fact that for $i\geq 2$
\begin{equation}\label{anpn}A_i(n)P_n^{[n]}=\left(\sum_{k= 0}^{i-2}\theta_{i-k-2}^2(n)(-\cl_n)^{k}(k+1)\right)P_n^{[n]}=(i-1)(-\cl_n)^{i-2}P_n^{[n]}.
\end{equation}

By induction assumption and Lemma \ref{omep}, we have for $i\geq 2$, $$\left(-\frac1d\theta_i^1(n)H\cl-\theta_i^1(n)H^2+(-\cl)^iH\right)P_n^{[n]}\in\Omega^{\bot}.$$  Let $i=j=2$ and by (\ref{g21tp}) we have 
$$\left(-\frac1d\theta_2^1(n)H\cl-\theta_2^1(n)H^2+(-\cl)^2H\right)P_n^{[n]}\in\Omega.$$
Therefore 
\begin{equation}\label{tpsp2}\left(-\frac1d\theta_2^1(n)H\cl-\theta_2^1(n)H^2+(-\cl)^2H\right)P_n^{[n]}=0.\end{equation}  Again let $j=2$ by (\ref{g21tp}) and (\ref{tpsp2}) we have
$$\left(-\frac1d\theta_i^1(n)H\cl-\theta_i^1(n)H^2+(-\cl)^iH\right)P_n^{[n]}\in\Omega.$$
Therefore 
\[\left(-\frac1d\theta_i^1(n)H\cl-\theta_i^1(n)H^2+(-\cl)^iH\right)P_n^{[n]}=0,\]
and also $$\left(\frac1d\theta_j^1(n+1)\theta_i^1(n+1)+\theta_{i+j-1}^1(n+1)\right)P_{n+1}^{[n+1]}=0.$$
The lemma is proved.
\end{proof}
\begin{lemma}\label{pnt32}In $H^*(S^{[n]},\bq)$ we have for all $i,j,i+k,j+k\geq 2$
\begin{eqnarray}\label{g22tp1}&&\left(\frac{j-1}{d}\theta_i^1(n)\theta_{j+k}^1(n)+\frac{i-1}{d}\theta_j^1(n)\theta_{i+k}^1(n)+(i+j-2)\theta_{i+j+k-1}^1(n)\right)P_{n-1}^{[n]}\nonumber\\
&=&\frac{2d-1}{d(d-2)}P_n^{[n]}\left[\frac{j-1}{d}\big(\theta_i^0(n)\theta_{j+k}^1+\theta_i^1(n)\theta_{j+k}^0(n)\big)+\right.\nonumber\\&&\left.+\frac{i-1}{d}\big(\theta_j^0(n)\theta_{i+k}^1(n)+\theta_{j}^1(n)\theta_{i+k}^0(n)\big)+(i+j-2)\theta_{i+j+k-1}^0(n)\right]\nonumber\\
&&+\frac3{d-2}(i+j-2)\theta_{i+j-2+k}^1(n)P_n^{[n]}\nonumber\\
&&-\frac3{d-2}(j-1)\big(\theta_i^1(n)\theta_{j+k-2}^2(n)+\theta_{i-2}^2(n)\theta_{j+k}^1(n)\big)P_n^{[n]}\nonumber\\
&&-\frac3{d-2}(i-1)\big(\theta_j^1(n)\theta_{i+k-2}^2(n)+\theta_{j-2}^2(n)\theta_{i+k}^1(n)\big)P_n^{[n]}\nonumber\\
\end{eqnarray}
In particular let $i,j\geq2,k=0$, then we have
\begin{eqnarray}\label{g22tp2}&&\left(\frac1d\theta_i^1(n)\theta_j^1(n)+\theta_{i+j-1}^1(n)\right)P_{n-1}^{[n]}
\nonumber\\&=&\frac{2d-1}{d(d-2)}\left(\frac1d(\theta_i^0(n)\theta_j^1(n)+\theta_j^0(n)\theta_i^1(n))+\theta_{i+j-1}^0(n)\right)P_n^{[n]}\nonumber\\
&&+\frac3{d-2}\theta_{i+j-2}^1(n)P_n^{[n]}-\frac3{d-2}\big(\theta_i^1(n)\theta_{j-2}^2(n)+\theta_{i-2}^2(n)\theta_{j}^1(n)\big)P_n^{[n]}\nonumber\\
\end{eqnarray}
\end{lemma}
The proof of Lemma \ref{pnt32} is in Appendix \ref{pl32}.  By (\ref{g22tp2}) for $i=j=2$ we have
\begin{eqnarray}\label{g22tp3}&&\left(\frac1d\theta_2^1(n)\theta_2^1(n)+\theta_{3}^1(n)\right)P_{n-1}^{[n]}
\nonumber\\&=&\frac{2d-1}{d(d-2)}\left(\frac1d(\theta_2^0(n)\theta_2^1(n)+\theta_2^0(n)\theta_2^1(n))+\theta_{3}^0(n)\right)P_n^{[n]}-\frac3{d-2}\theta_{2}^1(n)P_n^{[n]}\nonumber
\end{eqnarray}
Recall that 
\[\gamma_2(0)=\frac12\left(\left(d+\frac{2n}d\right)\theta_2^1(n)-\theta_3^0(n)\right).\]
Hence by Lemma \ref{pnt2} (2) we have
\begin{eqnarray}\label{sp22tp}-2\gamma_2(1)P_{n-1}^{[n]}&=&\left(\frac1d\theta_2^1(n)^2+\theta_3^1(n)\right)P_{n-1}^{[n]}-2\theta_2^1(n)P_n^{[n]}\nonumber\\
&=&\frac{2d-1}{d(d-2)}\left(-\frac{2n}d\theta_2^1(n)+\theta_3^0(n)\right)P_{n}^{[n]}-\frac{2d-1}{d-2}\theta_2^1(n)P_n^{[n]}\nonumber\\
&=&-\frac{2(2d-1)}{d(d-2)}\gamma_2(0)P_n^{[n]}\end{eqnarray}
\begin{proof}[Proof of Theorem \ref{main3} (2)]By Lemma \ref{pnt31} (1) and (\ref{sp22tp}) on $S^{[n]}\times\ls$ we have
\[\gamma_2(1)P^{[n]}(\xi)+\frac{2d-1}{d(d-2)}\xi \gamma_2(0)P^{[n]}(\xi)=\xi^2\cdot \mh_1,\]
for some $\mh_1\in A^*(S^{[n]}\times\ls)$.  Therefore
\[\gamma_2(1)P^{[n]}(\xi)+\frac{2d-1}{d(d-2)}\xi \gamma_2(0)P^{[n]}(\xi)\in \Ker(\xi^{\ttb+1-2})\]
and hence $\gamma_2(1)\in \widetilde{P}_2\Gamma^4$.  Since by (\ref{nog2}) $\frac{2d-1}{d(d-2)}\xi \gamma_2(0)P^{[n]}(\xi)\in \Ker(\xi^{\ttb+1-1})\setminus\Ker(\xi^{\ttb+1-2})$,  $\gamma_2(1)\not\in\widetilde{P}_1\Gamma^4$ and hence is almost of perversity 2.  
\end{proof}

\subsection{The computation for $\gamma_3(0)$.}
By (\ref{exfga}) we have\small
\[\gamma_3(0)=\frac16\left(\theta_4^0(n)+\frac3d\theta_2^1(n)\theta_3^0(n)-(\frac{3n}{d^2}+2)\theta_2^1(n)^2\right)+d\gamma_2(1)+\left(\frac n6-\frac{d^2}2\right)\theta_2^2(n).\]
\normalsize
Define 
$$\gamma'_3(0):=\theta_4^0(n)+\frac3d\theta_2^1(n)\theta_3^0(n)-(\frac{3n}{d^2}+2)\theta_2^1(n)^2+(n-3d^2)\theta_2^2(n).$$
Notice that we have already proved that $\gamma_2(1)$ is almost of perversity 2.  
Also by Lemma \ref{pnt2} (1) (2), we have 
\[\theta_2^2(n)P_n^{[n]}=0,~~\theta_2^2(n)P_{n-1}^{[n]}=\frac1d\theta_2^1(n)P_n^{[n]}.\]
Therefore to show $\gamma_3(0)$ is almost of perversity 3, it suffices to show
\begin{equation}\gamma'_3(0)P_n^{[n]}=0\text{ and }\gamma'_3(0)P_{n-1}^{[n]}\not\in P_n^{[n]}\Gamma^2.
\end{equation}
\begin{lemma}\label{pnt41}In $H^*(S^{[n]},\bq)$ we have for all $i,j,i+k,j+k\geq 2$
\begin{eqnarray}\label{g3tp1}&&\qquad(j-i)\big(\theta_{i+j+k-1}^0(n)-d\theta_{i+j+k-2}^1(n)\big)P_{n}^{[n]}=\nonumber\\
&&P_n^{[n]}\left[\frac{j-1}{d}\big(\theta_i^0(n)\theta_{j+k}^1(n)-\theta_i^1(n)\theta_{j+k}^0(n)\big)-\frac{i-1}{d}\big(\theta_j^0(n)\theta_{i+k}^1(n)-\theta_{j}^1(n)\theta_{i+k}^0(n)\big)\right]\nonumber\\&&+d(j-1)\big(\theta_i^1(n)\theta_{j+k-2}^2(n)-\theta_{i-2}^2(n)\theta_{j+k}^1(n)\big)P_n^{[n]}\nonumber\\
&&-d(i-1)\big(\theta_{j}^1(n)\theta_{i+k-2}^2(n)-\theta_{j-2}^2(n)\theta_{i+k}^1(n)\big)P_n^{[n]}\nonumber\\
\end{eqnarray}
In particular let $i,j\geq2,k=0$, then we have\small
\begin{eqnarray}\label{g3tp2}&&\qquad(j-i)\big(\theta_{i+j-1}^0(n)-d\theta_{i+j-2}^1(n)\big)P_{n}^{[n]}=\nonumber\\
&&P_n^{[n]}\left[\frac{i+j-2}{d}\big(\theta_i^0(n)\theta_{j}^1(n)-\theta_i^1(n)\theta_{j}^0(n)\big)+d(j+i-2)\big(\theta_i^1(n)\theta_{j-2}^2(n)-\theta_{i-2}^2(n)\theta_{j}^1(n)\big)\right]\nonumber\\\end{eqnarray}
\end{lemma}\normalsize
The proof of Lemma \ref{pnt41} is in Appendix \ref{pl41}.  

\begin{proof}[Proof of Theorem \ref{main3} (3)] By (\ref{g3tp2}) and Lemma \ref{pnt31} (1) let $i=2,j=3$ we have
\begin{eqnarray}\label{g3tp3}\theta_{4}^0(n)P_{n}^{[n]}&=&P_n^{[n]}\left(\frac{3}{d}\big(\theta_2^0(n)\theta_{3}^1(n)-\theta_2^1(n)\theta_{3}^0(n)\big)+2\theta_{2}^1(n)\theta_2^1(n)\right)\nonumber\\
&=&P_n^{[n]}\left(-\frac{3}{d}\theta_2^1(n)\theta_{3}^0(n)+\left(\frac{3n}{d^2}+2\right)\theta_{2}^1(n)\theta_2^1(n)\right)\nonumber \end{eqnarray}
Hence
\begin{equation}\label{sp3tp}\gamma'_3(0)P_{n}^{[n]}=0.\end{equation}
Therefore we only need to show $\gamma'_3(0)P_{n-1}^{[n]}\not\in P_n^{[n]}\Gamma^2.$

Assume $\gamma'_3(0)P_{n-1}^{[n]}\in P_n^{[n]}\Gamma^2$. Since $\Gamma^2=\bq\xi\oplus\bq\gamma_2(0)$, $\exists \lambda\in\bq$ such that $\gamma'_3(0)P_{n-1}^{[n]}=2\lambda \gamma_2(0)P_n^{[n]}$.  Then we have
\begin{eqnarray}\label{g3nz}\left(\theta_4^0(n)+\frac3d\theta_2^1(n)\theta_3^0(n)-\left(\frac{3n}d+2\right)\theta_2^1(n)^2\right)P_{n-1}^{[n]}&=&\lambda 
\left(\left(d+\frac{2n}d\right)\theta_2^1(n)-\theta_3^0(n)\right)P_n^{[n]}\nonumber\\
&&-\frac1d(n-3d^2)\theta_2^1(n)P_n^{[n]}.\end{eqnarray}

We replace $n$ by $n+1$ and pull back (\ref{g3nz}) to $\p(\mi_n)$.  By a direct computation we have
\begin{eqnarray}\label{abc234}&&A_2(n)=\mathbf{1},B_2(n)=C_2(n)=0;\nonumber\\
&&A_3(n)=2(-\cl_n),B_3(n)=3\cdot\mathbf{1},C_3(n)=0;\nonumber\\
&&A_4(n)=3(-\cl_n)^2+\theta_2^2(n),B_4(n)=9(-\cl_n)+\theta_2^1(n),C_4(n)=(6-n)\cdot\mathbf{1}.\nonumber\\
\end{eqnarray}
Hence\small
\begin{eqnarray}&&LHS=\left(\theta_4^0(n+1)+\frac3d\theta_2^1(n+1)\theta_3^0(n+1)-\left(\frac{3(n+1)}d+2\right)\theta_2^1(n+1)^2\right)P_{n}^{[n+1]}\nonumber\\
&=&\left(\theta_4^0(n)+\frac3d\theta_2^1(n)\theta_3^0(n)-\left(\frac{3(n+1)}d+2\right)\theta_2^1(n)^2\right)P_{n-1}^{[n]}(\cl_n+dH)\nonumber\\&&+\left(\theta_4^0(n)+\frac3d\theta_2^1(n)\theta_3^0(n)-\left(\frac{3(n+1)}d+2\right)\theta_2^1(n)^2\right)P_{n}^{[n]}\nonumber\\
&&-\left(A_4(n)+\frac3d\theta_2^1(n)A_3(n)\right)(P_n^{[n]}+P_{n-1}^{[n]}(\cl_n+dH))\nonumber\\
&&-\left(B_4(n)+\frac3d\theta_2^1(n)B_3(n)-A_3(n)A_2(n)+\frac3d\theta_3^0(n)A_2(n)\right)H(P_n^{[n]}+P_{n-1}^{[n]}(\cl_n+dH))\nonumber\\
&&-2\left(\frac{3(n+1)}d+2\right)A_2(n)\theta_2^1(n)H(P_n^{[n]}+P_{n-1}^{[n]}(\cl_n+dH))\nonumber\\
&&-\left(C_4(n)-A_2(n)B_3(n)+\left(\frac{3(n+1)}d+2\right)A_2(n)^2\right)H^2(P_n^{[n]}+P_{n-1}^{[n]}(\cl_n+dH)).\nonumber\\\end{eqnarray}
By (\ref{abc234}) we have \begin{eqnarray}
&&LHS=\left(\theta_4^0(n)+\frac3d\theta_2^1(n)\theta_3^0(n)-\left(\frac{3(n+1)}d+2\right)\theta_2^1(n)^2\right)P_{n-1}^{[n]}(\cl_n+dH)\nonumber\\&&+\left(\theta_4^0(n)+\frac3d\theta_2^1(n)\theta_3^0(n)-\left(\frac{3(n+1)}d+2\right)\theta_2^1(n)^2\right)P_{n}^{[n]}\nonumber\\
&&-\left(3\cl_n^2+\theta_2^2(n)+\frac6d\theta_2^1(n)(-\cl_n)\right)(P_n^{[n]}+P_{n-1}^{[n]}(\cl_n+dH))\nonumber\\
&&-\left(7(-\cl_n)+\theta_2^1(n)+\frac9d\theta_2^1(n)+\frac3d\theta_3^0(n)\right)H(P_n^{[n]}+P_{n-1}^{[n]}(\cl_n+dH))\nonumber\\
&&-2\left(\frac{3(n+1)}d+2\right)\theta_2^1(n)H(P_n^{[n]}+P_{n-1}^{[n]}(\cl_n+dH))\nonumber\\
&&-\left(\frac{3(n+1)}d-n+5\right)H^2(P_n^{[n]}+P_{n-1}^{[n]}(\cl_n+dH)).\nonumber\\
\end{eqnarray}\normalsize
\begin{eqnarray}\label{g20p}&&RHS=\lambda 
\left(\left(d+\frac{2(n+1)}d\right)\theta_2^1(n+1)-\theta_3^0(n+1)\right)P_{n+1}^{[n+1]}\nonumber\\
&&-\frac1d(n-3d^2)\theta_2^1(n+1)P_{n+1}^{[n+1]}\nonumber\\
&=&\lambda\left(\left(d+\frac{2(n+1)}d\right)\theta_2^1(n)-\theta_3^0(n)\right)P_{n}^{[n]}(\cl_n+dH)\nonumber\\
&&-\lambda\left(\left(d+\frac{2(n+1)}d\right)H+2\cl_n-3H\right)P_{n}^{[n]}(\cl_n+dH)\nonumber\\
&&-\frac1d(n-3d^2)(\theta_2^1(n)-H)P_n^{[n]}(\cl_n+dH)\nonumber\\
\end{eqnarray}

Then we must have for any $t,j\geq0$
\[(\phi_n)_*(LHS\cdot(-\cl_n)^tH^j)=(\phi_n)_*(RHS\cdot(-\cl_n)^tH^j).\]
Let $t=0,j=0$.  Then we have
\begin{eqnarray}\label{cm4lr}
(\phi_n)_*(LHS)&=&-(8d+6n)(\frac{1}d+1)\theta_2^1(n)P_{n-1}^{[n]}+(4n-\frac{3(n+1)}d-5)P_n^{[n]};\nonumber\\
&=&\big(\lambda(-d^2-2+3d)+(n-3d^2)\big)P_n^{[n]}=(\phi_n)_*(RHS).\nonumber\\
\end{eqnarray}
By (\ref{cm4lr}) we have $\theta_2^1(n)P_{n-1}^{[n]}=\lambda'P_n^{[n]}$ for some $\lambda'\in \bq$.  Then by Lemma \ref{pnt2} (1) (2) we must have
\begin{equation}0=\lambda'\theta_2^2(n)P_n^{[n]}=\theta_2^2(n)\theta_2^1(n)P_{n-1}^{[n]}=\frac1d\theta_2^1(n)^2P_n^{[n]},\end{equation}
which contradicts to Lemma \ref{inp13} (1) for $n\geq 2$.  Therefore $\gamma_3(0)$ is of almost perversity 3.
\end{proof}

\appendix
\section{The proof of Lemma \ref{pnt32}.}\label{pl32}
Recall that we want to prove that in $H^*(S^{[n]},\bq)$ for all $i,j,i+k,j+k\geq2$
\begin{eqnarray}\label{Ag22tp1}&&\left(\frac{j-1}{d}\theta_i^1(n)\theta_{j+k}^1(n)+\frac{i-1}{d}\theta_j^1(n)\theta_{i+k}^1(n)+(i+j-2)\theta_{i+j+k-1}^1(n)\right)P_{n-1}^{[n]}\nonumber\\
&=&\frac{2d-1}{d(d-2)}P_n^{[n]}\left[\frac{j-1}{d}\big(\theta_i^0(n)\theta_{j+k}^1(n)+\theta_i^1(n)\theta_{j+k}^0(n)\big)+\right.\nonumber\\&&\left.+\frac{i-1}{d}\big(\theta_j^0(n)\theta_{i+k}^1(n)+\theta_{j}^1(n)\theta_{i+k}^0(n)\big)+(i+j-2)\theta_{i+j+k-1}^0(n)\right]\nonumber\\
&&+\frac3{d-2}(i+j-2)\theta_{i+j-2+k}^1(n)P_n^{[n]}\nonumber\\
&&-\frac3{d-2}(j-1)\big(\theta_i^1(n)\theta_{j+k-2}^2(n)+\theta_{i-2}^2(n)\theta_{j+k}^1(n)\big)P_n^{[n]}\nonumber\\
&&-\frac3{d-2}(i-1)\big(\theta_j^1(n)\theta_{i+k-2}^2(n)+\theta_{j-2}^2(n)\theta_{i+k}^1(n)\big)P_n^{[n]}\nonumber\\
\end{eqnarray}
In particular let $i,j\geq2,k=0$, then we have
\begin{eqnarray}\label{Ag22tp2}&&\left(\frac1d\theta_i^1(n)\theta_j^1(n)+\theta_{i+j-1}^1(n)\right)P_{n-1}^{[n]}
\nonumber\\&=&\frac{2d-1}{d(d-2)}\left(\frac1d(\theta_i^0\theta_j^1(n)+\theta_j^0(n)\theta_i^1(n))+\theta_{i+j-1}^0(n)\right)P_n^{[n]}\nonumber\\
&&+\frac3{d-2}\theta_{i+j-2}^1(n)P_n^{[n]}-\frac3{d-2}\big(\theta_i^1(n)\theta_{j-2}^2(n)+\theta_{i-2}^2(n)\theta_{j}^1(n)\big)P_n^{[n]}\nonumber\\
\end{eqnarray}

We do induction on $n$.  When $n=1$, by dimension reason we only need to check (\ref{g22tp1}) for $i=j=2,k=0$.  Since $\theta_2^0(1)=-\mathbf{1}$, $\theta_2^1(1)=-H$, $\theta_3^0(1)=-3H$, $\theta_3^1(1)=-3H^2$, $P_0^{[1]}=\mathbf{1}$ and $P_1^{[1]}=dH$, we have for (\ref{g22tp1})
\begin{eqnarray}RHS&=&\left(\frac{2d-1}{(d-2)}\left(\frac{4}{d}-6\right)+\frac{6d}{d-2}\right)H^2\nonumber\\
&=&\frac{4(2d-1)-6d(2d-1)+6d^2}{d(d-2)}H^2=\frac{-6d^2+14d-4}{d(d-2)}H^2\nonumber\\
&=&\frac{-2(3d-1)(d-2)}{d(d-2)}H^2=2\left(\frac1d-3\right)H^2=LHS.\nonumber\end{eqnarray}

Now we assume that (\ref{Ag22tp1}) holds for $n$.  We will compare the pullbacks of LHS and RHS of (\ref{Ag22tp1}) to $\p(\mi_n)$ via $\psi_n:\p(\mi_n)\ra S^{[n+1]}$.
\subsection{The computation on the LHS}  We first deal with LHS of (\ref{Ag22tp1}).  On $\p(\mi_n)$ we have \footnotesize
\begin{eqnarray}&&\left(\frac{j-1}{d}\theta_i^1(n+1)\theta_{j+k}^1(n+1)+\frac{i-1}{d}\theta_j^1(n+1)\theta_{i+k}^1(n+1)+(i+j-2)\theta_{i+j+k-1}^1(n+1)\right)P_{n}^{[n+1]}\nonumber\\
&=&\left[\frac{j-1}{d}(\theta_i^1(n)-A_i(n)H-B_i(n)H^2)(\theta_{j+k}^1(n)-A_{j+k}(n)H-B_{j+k}(n)H^2)\right.\nonumber\\
&&+\frac{i-1}{d}(\theta_j^1(n)-A_{j}(n)H-B_{j}(n)H^2)(\theta_{i+k}^1(n)-A_{i+k}(n)H-B_{i+k}(n)H^2)\nonumber\\
&&\left.+(i+j-2)(\theta_{i+j+k-1}^1(n)-A_{i+j+k-1}(n)H-B_{i+j+k-1}(n)H^2)\right](P_{n}^{[n]}+P_{n-1}^{[n]}(\cl_n+dH))
\nonumber\\&=&\left(\frac{j-1}{d}\theta_i^1(n)\theta_{j+k}^1(n)+\frac{i-1}{d}\theta_j^1(n)\theta_{i+k}^1(n)+(i+j-2)\theta_{i+j+k-1}^1(n)\right)P_{n-1}^{[n]}(\cl_n+dH)+I+II,\nonumber
\end{eqnarray}
\normalsize
where 
\begin{eqnarray}&&I:=\left[\frac{j-1}{d}(\theta_i^1(n)-A_i(n)H-B_i(n)H^2)(\theta_{j+k}^1(n)-A_{j+k}(n)H-B_{j+k}(n)H^2)\right.\nonumber\\
&&+\frac{i-1}{d}(\theta_j^1(n)-A_{j}(n)H-B_{j}(n)H^2)(\theta_{i+k}^1(n)-A_{i+k}(n)H-B_{i+k}(n)H^2)\nonumber\\
&&\left.+(i+j-2)(\theta_{i+j+k-1}^1(n)-A_{i+j+k-1}(n)H-B_{i+j+k-1}(n)H^2)\right]P_{n-1}^{[n]}(\cl_n+dH)\nonumber\\
&&-\left(\frac{j-1}{d}\theta_i^1(n)\theta_{j+k}^1(n)+\frac{i-1}{d}\theta_j^1(n)\theta_{i+k}^1(n)+(i+j-2)\theta_{i+j+k-1}^1(n)\right)P_{n-1}^{[n]}(\cl_n+dH),\nonumber\end{eqnarray}
and
\begin{eqnarray}II&:=&\left[\frac{j-1}{d}(\theta_i^1(n)-A_i(n)H-B_i(n)H^2)(\theta_{j+k}^1(n)-A_{j+k}(n)H-B_{j+k}(n)H^2)\right.\nonumber\\
&&+\frac{i-1}{d}(\theta_j^1(n)-A_{j}(n)H-B_{j}(n)H^2)(\theta_{i+k}^1(n)-A_{i+k}(n)H-B_{i+k}(n)H^2)\nonumber\\
&&\left.+(i+j-2)(\theta_{i+j+k-1}^1(n)-A_{i+j+k-1}(n)H-B_{i+j+k-1}(n)H^2)\right]P_{n}^{[n]}.\nonumber\end{eqnarray}

By Lemma \ref{pnt2} and a direct computation we have\small
\begin{eqnarray}-I&=&\frac{j-1}{d}\big(\theta_i^1(n)A_{j+k}+\theta_{j+k}^1(n)A_i(n)+dA_{i+j+k-1}(n)\big)HP_{n-1}^{[n]}(\cl_n+dH)\nonumber\\
&&+\frac{i-1}{d}\big(\theta_j^1(n)A_{i+k}+\theta_{i+k}^1(n)A_j(n)+dA_{i+j+k-1}(n)\big)HP_{n-1}^{[n]}(\cl_n+dH)\nonumber\\
&&+\frac{j-1}{d^2}\big(\theta_i^1(n)B_{j+k}(n)+\theta_{j+k}^1(n)B_i(n)+dB_{i+j+k-1}(n)-A_{j+k}(n)A_{i}(n)\big)P_{n}^{[n]}H\cl_n\nonumber\\
&&+\frac{i-1}{d^2}\big(\theta_j^1(n)B_{i+k}(n)+\theta_{i+k}^1(n)B_i(n)+dB_{i+j+k-1}(n)-A_{i+k}(n)A_{j}(n)\big)P_{n}^{[n]}H\cl_n\nonumber\\
\end{eqnarray}\normalsize
By Lemma \ref{pnt31} (1) and a direct computation we have\small
\begin{eqnarray}-II&=&\frac{j-1}{d}(\theta_{j+k}^1(n)A_i(n)+\theta_i^1(n)A_{j+k}(n)+dA_{i+j+k-1}(n))HP_n^{[n]}\nonumber\\
&&+\frac{i-1}{d}(\theta_{i+k}^1(n)A_j(n)+\theta_j^1(n)A_{i+k}(n)+dA_{i+j+k-1}(n))HP_n^{[n]}\nonumber\\
&&+\frac{j-1}{d}(\theta_{j+k}^1(n)B_i(n)+\theta_i^1(n)B_{j+k}(n)+dB_{i+j+k-1}(n)-A_i(n)A_{j+k}(n))H^2P_n^{[n]}\nonumber\\
&&+\frac{i-1}{d}(\theta_{i+k}^1(n)B_j(n)+\theta_j^1(n)B_{i+k}(n)+dB_{i+j+k-1}(n)-A_j(n)A_{i+k}(n))H^2P_n^{[n]}.\nonumber\end{eqnarray}

\normalsize
By (\ref{defabc}) and Lemma \ref{pnt2} for $i\geq 2$ we have
\begin{equation}\label{bnpn}B_i(n)P_{n}^{[n]}=3\binom{i-1}{2}(-\cl_n)^{i-3}P_n^{[n]}+\sum_{k=0}^{i-4}\theta_{i-2-k}^1(n)(-\cl_n)^k(k+1)P_n^{[n]},\end{equation}
\begin{equation}\label{anpn1}A_i(n)P_{n-1}^{[n]}=(-\cl_n)^{i-2}(i-1)P_{n-1}^{[n]}+\frac1d\sum_{k=0}^{i-4}\theta_{i-2-k}^1(n)(-\cl_n)^k(k+1)P_n^{[n]}.\end{equation}
By Lemma \ref{pnt2} and Lemma \ref{pnt31}, we have on $\p(\mi_n)$
\begin{equation}\label{svopn}H^2P_n^{[n]}(\cl_n+dH)=0,~~\theta_i^1(n)HP_n^{[n]}(\cl_n+dH)=d(-\cl_n)^iHP_n^{[n]}\text{ for }i\geq 2\end{equation}
Hence by Lemma \ref{pnt31} we have 
\begin{eqnarray}\label{bnpnh}&&B_i(n)HP_{n}^{[n]}(\cl_n+dH)\nonumber\\
&=&3\binom{i-1}{2}(-\cl_n)^{i-3}HP_n^{[n]}(\cl_n+dH)+d\binom{i-2}2(-\cl_n)^{i-2}HP_n^{[n]},\nonumber\\\end{eqnarray}
\begin{eqnarray}\label{anpn1h}&&A_i(n)HP_{n-1}^{[n]}(\cl_n+dH)\nonumber\\
&=&(-\cl_n)^{i-2}(i-1)HP_{n-1}^{[n]}(\cl_n+dH)+\binom{i-2}2(-\cl_n)^{i-2}HP_n^{[n]}.\nonumber\\\end{eqnarray}

By (\ref{anpn1h}) and (\ref{svopn}) we have
\begin{eqnarray}\label{anpn2}&&\big(\theta_i^1(n)A_{j+k}(n)H+\theta_{j+k}^1(n)A_i(n)H+dA_{i+j+k-1}(n)H\big)P_{n-1}^{[n]}(\cl_n+dH)\nonumber\\
&=&(\theta_i^1(n)(-\cl_n)^{j+k-2}(j+k-1)H+\theta_{j+k}^1(n)(-\cl_n)^{i-2}(i-1)H)P_{n-1}^{[n]}(\cl_n+dH)\nonumber\\
&&+d(-\cl_n)^{i+j+k-3}(i+j+k-2)HP_{n-1}^{[n]}(\cl_n+dH)\nonumber\\
&&+d\big((i-1)(j+k-1)-1\big)(-\cl_n)^{i+j+k-3}HP_n^{[n]}\nonumber\\
\end{eqnarray}
Notice that for $s,t\geq 2$ we always have
\begin{equation}\label{hp23}\theta_s(n)\binom{t-2}{2}(-\cl_n)^{t-2}HP_n^{[n]}=-\binom{t-2}{2}(-\cl_n)^{t+s-3}HP_n^{[n]},\end{equation}
since they are equal by (\ref{svopn}) for $t\neq 2,3$ and both zero for $t=2,3$.

By (\ref{anpn}) we have
\begin{eqnarray}\label{aanpn1}A_{j+k}(n)A_i(n)HP_{n}^{[n]}(\cl_n+dH)=(i-1)(j+k-1)(-\cl_n)^{i+j+k-4}HP_{n}^{[n]}(\cl_n+dH).\nonumber\\
\end{eqnarray}
By (\ref{bnpnh}), (\ref{svopn}) and (\ref{hp23}) we have
\begin{eqnarray}\label{bnpn2}&&\quad\big(\theta_i^1(n)B_{j+k}(n)+\theta_{j+k}^1(n)B_i(n)+dB_{i+j+k-1}(n)\big)HP_{n}^{[n]}(\cl_n+dH)\nonumber\\
=&&(-3d(i-1)(j+k-1)+3d^2\binom{i+j+k-2}2(-\cl_n)^{i+j+k-4}H^2P_{n}^{[n]}
\nonumber\\
&&+d^2\big((i-1)(j+k-1)-1\big)(-\cl_n)^{i+j+k-3}HP_n^{[n]}\nonumber\\
\end{eqnarray}
Define 
\begin{eqnarray}\label{defefg}\mf(i;s)&:=&\big(\theta_i^1(n)A_{j+k}(n)H+\theta_{j+k}^1(n)A_i(n)H+dA_{i+j+k-1}(n)H\big)P_{n-1}^{[n]}(\cl_n+dH)\nonumber\\
&&-d\big((i-1)(j+k-1)-1\big)(-\cl_n)^{i+j+k-3}HP_n^{[n]}\nonumber\\
&=&\big(\theta_i^1(n)(-\cl_n)^{s-2}(s-1)+\theta_{s}^1(n)(-\cl_n)^{i-2}(i-1)\big)HP_{n-1}^{[n]}(\cl_n+dH)\nonumber\\
&&+d(-\cl_n)^{i+s-3}(i+s-2)HP_{n-1}^{[n]}(\cl_n+dH);\nonumber\\
\mh(i;s)&:=&\big(\theta_i^1(n)A_s(n)+\theta_s^1(n)A_i(n)+dA_{i+s-1}(n)\big)HP_n^{[n]}\nonumber\\
&=&\big(\theta_i^1(n)(-\cl_n)^{s-2}(s-1)+\theta_{s}^1(n)(-\cl_n)^{i-2}(i-1)\big)HP_{n}^{[n]}\nonumber\\
&&+d(-\cl_n)^{i+s-3}(i+s-2)HP_{n}^{[n]}.\nonumber
\end{eqnarray}

By (\ref{anpn}) (\ref{anpn2}), (\ref{aanpn1}) and (\ref{bnpn2}) we have
\small
\begin{eqnarray}\label{ipii}
-I-II&=&\frac{j-1}d(\mf(i;j+k)+\mh(i;j+k))+\frac{i-1}d(\mf(j;i+k)+\mh(j;i+k))\nonumber\\
&&+(i-1)(j-1)(i+j+2k-2)\left(2-\frac 3d+\frac1{d^2}\right)(-\cl_n)^{i+j+k-3}HP_{n}^{[n]}
\nonumber\\&&-2(i+j-2)(-\cl_n)^{i+j+k-3}HP_{n}^{[n]}
\nonumber\\&&+3(i+j-2)\binom{i+j+k-2}2(-\cl_n)^{i+j+k-4}H^2P_{n}^{[n]}\nonumber\\
&&-\frac1d(i-1)(j-1)(i+j+2k-2)(-\cl_n)^{i+j+k-4}H^2P_{n}^{[n]}.\nonumber\\
\end{eqnarray}
\normalsize
\subsection{The computation on the RHS}  Now we deal with the RHS of (\ref{Ag22tp1}).  On $\p(\mi_n)$ by (\ref{svopn}) we have
\begin{eqnarray}&&P_{n+1}^{[n+1]}\left[\frac{j-1}{d}\big(\theta_i^0(n+1)\theta_{j+k}^1(n+1)+\theta_i^1(n+1)\theta_{j+k}^0(n+1)\big)+\right.\nonumber\\&&\left.+\frac{i-1}{d}\big(\theta_j^0(n+1)\theta_{i+k}^1(n+1)+\theta_{j}^1(n+1)\theta_{i+k}^0(n+1)\big)+(i+j-2)\theta_{i+j+k-1}^0(n+1)\right]\nonumber\\
&=&P_{n}^{[n]}(\cl_n+dH)\left[\frac{j-1}{d}\big(\theta_i^0(n)\theta_{j+k}^1(n)+\theta_i^1(n)\theta_{j+k}^0(n)\big)+\right.\nonumber\\&&\left.+\frac{i-1}{d}\big(\theta_j^0(n)\theta_{i+k}^1(n)+\theta_{j}^1(n)\theta_{i+k}^0(n)\big)+(i+j-2)\theta_{i+j+k-1}^0(n)\right]+III,\nonumber\\
\end{eqnarray}
where
\begin{eqnarray}
&&-III:=P_n^{[n]}(\cl_n+dH)\left[\frac{j-1}{d}\big(A_i(n)\theta_{j+k}^1(n)+\theta_i^1(n)A_{j+k}(n)+dA_{i+j+k-1}(n)\big)+\right.\nonumber\\&&\left.+\frac{i-1}{d}\big(A_j(n)\theta_{i+k}^1(n)+\theta_{j}^1(n)A_{i+k}(n)+dA_{i+j+k-1}(n)\big)\right]\nonumber\\
&&+P_n^{[n]}(\cl_n+dH)H\left[\frac{j-1}{d}\big(A_{j+k}(n)\theta_i^0(n)+A_i(n)\theta_{j+k}^0(n)-2A_i(n)A_{j+k}(n)\big)
\right.\nonumber\\ &&+\frac{i-1}{d}\big(A_{i+k}(n)\theta_j^0(n)+A_j(n)\theta_{i+k}^0(n)-2A_j(n)A_{i+k}(n)\big)\nonumber\\ &&+\frac{j-1}{d}\big(B_i(n)\theta_{j+k}^1(n)+B_{j+k}(n)\theta_i^1(n)+dB_{i+j+k-1}(n)\big)
\nonumber\\&&\left.\frac{i-1}{d}\big(B_j(n)\theta_{i+k}^1(n)+B_{i+k}(n)\theta_j^1(n)+dB_{i+j+k-1}(n)\big)\right]\nonumber\\
\end{eqnarray}
Define
\begin{eqnarray}\label{defjk}\mj(i;s)&:=&(\theta_i^1(n)A_s(n)+\theta_{s}^1(n)A_i(n)+dA_{i+s-1}(n))P_{n}^{[n]}(\cl_n+dH)\nonumber\\
&=&(\theta_i^1(n)(-\cl_n)^{s-2}(s-1)+\theta_{s}^1(n)(-\cl_n)^{i-2}(i-1))P_{n}^{[n]}(\cl_n+dH)\nonumber\\
&&+d(-\cl_n)^{i+s-3}(i+s-2)P_{n}^{[n]}(\cl_n+dH);\nonumber\\
\mk(i;s)&:=&\big(\theta_i^0(n)A_s(n)+\theta_{s}^0(n)A_i(n)\big)HP_{n}^{[n]}(\cl_n+dH)\nonumber\\
&=&\big(\theta_i^0(n)(-\cl_n)^{s-2}(s-1)+\theta_{s}^0(n)(-\cl_n)^{i-2}(i-1)\big)HP_{n}^{[n]}(\cl_n+dH).\nonumber\\
\end{eqnarray}
By (\ref{anpn}) and (\ref{bnpn2}) we have

\begin{eqnarray}\label{iii}
-III&=&\frac{j-1}d(\mj(i;j+k)+\mk(i;j+k))+\frac{i-1}d(\mj(j;i+k)+\mk(j;i+k))\nonumber\\
&+&\frac1d(i-1)(j-1)(i+j+2k-2)\left(d^2-3d+2\right)(-\cl_n)^{i+j+k-3}HP_{n}^{[n]}
\nonumber\\&-&d(i+j-2)(-\cl_n)^{i+j+k-3}HP_{n}^{[n]}
\nonumber\\&+&3d(i+j-2)\binom{i+j+k-2}2(-\cl_n)^{i+j+k-4}H^2P_{n}^{[n]}\nonumber\\
&-&2(j-1)(i-1)(j+i+2k-2)(-\cl_n)^{i+j+k-4}H^2P_{n}^{[n]}.\nonumber\\ \end{eqnarray}

For the rest terms on RHS, by (\ref{anpn}) and (\ref{svopn}) we have on $\p(\mi_n)$
\begin{eqnarray}\label{strhs}&&(i+j-2)\theta_{i+j+k-2}^1(n+1)P_{n+1}^{[n+1]}\nonumber\\
&=&(i+j-2)\theta_{i+j+k-2}^1(n)P_{n}^{[n]}(\cl_n+dH)\nonumber\\
&&-(i+j-2)(i+j+k-3)(-\cl_n)^{i+j+k-4}HP_{n}^{[n]}(\cl_n+dH)\nonumber\\
\end{eqnarray}
and
\begin{eqnarray}\label{strhs2}&&\big(\theta_i^1(n+1)\theta_{j+k-2}^2(n+1)+\theta_{i-2}^2(n+1)\theta_{j+k}^1(n+1)\big)P_{n+1}^{[n+1]}\nonumber\\
&=&(\theta_i^1(n)\theta_{j+k-2}^2(n)+\theta_{i-2}^2(n)\theta_{j+k}^1(n))P_{n}^{[n]}(\cl_n+dH)+\me(i;j+k),\nonumber\\
\end{eqnarray}
where for $i,s\geq 2$
\begin{eqnarray}
\me(i;s)&:=&\big(\theta_{s-2}^2(n)(-\cl)^{i-1}(i-1)+\theta_{i-2}^2(n)(-\cl)^{s-1}(s-1)\big)HP_n^{[n]}\nonumber\\
&&-d\big(\theta_{s-2}^2(n)(-\cl_n)^{i-2}(i-1)+\theta_{i-2}^2(n)(-\cl_n)^{s-2}(s-1)\big)H^2P_n^{[n]};\nonumber\\
\end{eqnarray}
\begin{lemma}\label{meeq}For $i,s\geq 2$ we have
\[(\theta_{s-2}^2(n)(-\cl_n)^{i-2}(i-1)+\theta_{i-2}^2(n)(-\cl_n)^{s-2}(s-1))H^2P_n^{[n]}=2(i+s-3)(-\cl)^{i+s-4}H^2P_n^{[n]}.\]
\end{lemma}
Therefore we have
\begin{eqnarray}\label{eprme}
\me_(i;s)&=&\big(\theta_{s-2}^2(n)(-\cl)^{i-1}(i-1)+\theta_{i-2}^2(n)(-\cl)^{s-1}(s-1)\big)HP_n^{[n]}\nonumber\\
&&-2d(i+s-3)(-\cl)^{i+s-4}H^2P_n^{[n]}.\nonumber\\
\end{eqnarray}
\begin{proof}Both terms are zero except for $i=s=2$ when both terms equal $2H^2P_n^{[n]}$.  We are done.  
\end{proof}
\subsection{The proof}
We write $\ma\equiv \mb$ for $\ma,\mb\in H^*(\p(\mi_n),\bq)$ if $\ma-\mb\in\Omega^{\bot}$.  Define $\bar{d}:=\frac{2d-1}{d(d-2)}$, then $\frac{3}{d-2}=d\bar{d}-2$.  By induction assumption we only need to prove 
\begin{eqnarray}\label{eqiii}I+II&\equiv& \bar{d}III-(d\bar{d}-2)(i+j-2)(i+j+k-3)(-\cl_n)^{i+j+k-4}HP_{n}^{[n]}(\cl_n+dH)\nonumber\\&& -(d\bar{d}-2)\big((j-1)\me(i;j+k)+(i-1)\me(j;i+k)\big).\nonumber\\\end{eqnarray}
Define
\begin{eqnarray}\rl(i;s)&:=&\mf(i;s)+\mh(i;s)-2d(-\cl_n)^{i+s-3}HP_{n}^{[n]}\nonumber\\
&&+d(i-1)(s-1)\left(2-\frac 3d+\frac1{d^2}\right)(-\cl_n)^{i+s-3}HP_{n}^{[n]}
\nonumber\\&&+3d\binom{i+s-2}2(-\cl_n)^{i+s-4}H^2P_{n}^{[n]}\nonumber\\
&&-(i-1)(s-1)(-\cl_n)^{i+s-4}H^2P_{n}^{[n]}.\nonumber\\
\rr(i;s)&:=&\mj(i;s)+\mk(i;s)-d^2(-\cl_n)^{i+s-3}HP_{n}^{[n]}\nonumber\\
&+&(i-1)(s-1)\left(d^2-3d+2\right)(-\cl_n)^{i+s-3}HP_{n}^{[n]}
\nonumber\\&+&3d^2\binom{i+s-2}2(-\cl_n)^{i+s-4}H^2P_{n}^{[n]}\nonumber\\
&-&2d(i-1)(s-1)(-\cl_n)^{i+s-4}H^2P_{n}^{[n]}.\nonumber\\
\end{eqnarray}
We have 
\[-I-II=\frac{j-1}d\rl(i;j+k)+\frac{i-1}d\rl(j;i+k);\]
\[-III=\frac{j-1}d\rr(i;j+k)+\frac{i-1}d\rl(j;i+k).\]
The following lemma implies (\ref{eqiii}) and we are done.
\begin{lemma}\label{le1ap}For $i,s\geq 2$, we have $$\rl(i;s)\equiv\bar{d}\rr(i;s)+(d\bar{d}-2)d\big((i+s-3)(-\cl_n)^{i+s-4}HP_n^{[n]}(\cl_n+dH)+\me(i;s)\big).$$
\end{lemma}
\begin{proof}Notice that $d(2-\frac3d+\frac1{d^2})=(d^2-3d+2)\bar{d}$.  By a direct computation we have\small
\begin{eqnarray}\tw(i;s)&:=&\rl(i;s)-\bar{d}\rr(i;s)-(d\bar{d}-2)d\big((i+s-3)(-\cl_n)^{i+s-4}HP_n^{[n]}(\cl_n+dH)+\me(i;s))\nonumber\\
&=&\mf(i;s)+\mh(i;s)-\bar{d}(\mj(i;s)+\mk(i;s))
\nonumber\\
&&+(d\bar{d}-2)d(i+s-2)(-\cl_n)^{i+s-3}HP_n^{[n]}\nonumber\\
&&-(d\bar{d}-2)d\big(\theta_{s-2}^2(n)(-\cl)^{i-1}(i-1)+\theta_{i-2}^2(n)(-\cl)^{s-1}(s-1)\big)HP_n^{[n]}\nonumber\\
&&+3d(1-d\bar{d})\binom{i+s-2}2(-\cl_n)^{i+s-4}H^2P_{n}^{[n]}\nonumber\\
&&+(2d\bar{d}-1)(i-1)(s-1)(-\cl_n)^{i+s-4}H^2P_{n}^{[n]}\nonumber\\
&&+(d\bar{d}-2)d^2(i+s-3)(-\cl_n)^{i+s-4}H^2P_n^{[n]}.\nonumber\\
\end{eqnarray}\normalsize
We have $(-\cl_n)^{i+s-4}H^2P_{n}^{[n]}=0$ except for $i=s=2$ and for $i=s=2$
\begin{eqnarray}&&3d(1-d\bar{d})\binom{i+s-2}{2}+(2d\bar{d}-1)(i-1)(s-1)+(d\bar{d}-2)d^2(i+s-3)
\nonumber\\
&=&3d-3d^2\bar{d}+2d\bar{d}-1+d^3\bar{d}-2d^2=-(2d^2-3d+1)+d\bar{d}(d^2-3d+2)=0.\nonumber
\end{eqnarray}
Therefore
\begin{eqnarray}\tw(i;s)&=&\mf(i;s)+\mh(i;s)-\bar{d}(\mj(i;s)+\mk(i;s))
\nonumber\\
&&+(d\bar{d}-2)d(i+s-2)(-\cl_n)^{i+s-3}HP_n^{[n]}\nonumber\\
&&-(d\bar{d}-2)d\big(\theta_{s-2}^2(n)(-\cl)^{i-1}(i-1)+\theta_{i-2}^2(n)(-\cl)^{s-1}(s-1)\big)HP_n^{[n]}\nonumber\\
\end{eqnarray}
By Lemma \ref{omep} it is enough to show $$(\phi_n)_*(\tw(i;s) (-\cl_n)^jH^i)=0,\text{ for all }i,j\geq0.$$  

Firstly since $i,s\geq 2$, by (\ref{svopn}) and Lemma \ref{pnt2} we have
\begin{eqnarray}
\tw(i;s)\cdot H&=&(\mf(i;s)+\mh(i;s))H\nonumber\\
&=&\frac1d\big(\theta_i^1(n)(-\cl_n)^{s-2}(s-1)+\theta_{s}^1(n)(-\cl_n)^{i-2}(i-1)\big)HP_{n}^{[n]}(\cl_n+dH)\nonumber\\
&&+(-\cl_n)^{i+s-3}(i+s-2)HP_{n}^{[n]}(\cl_n+dH)=0\nonumber\\
\end{eqnarray}

Now we only need to show  
\[(\phi_n)_*(\tw(i;s)(-\cl_n)^j)=0,\forall~j\geq 0.\]
For $t\geq 0,s+t\geq3,i+t\geq 3$ by Lemma \ref{pnt2} and (\ref{psl}) we have 
\begin{eqnarray}
&&(\phi_n)_*((\mf(i;s)+\mh(i;s))(-
\cl_n)^t)\nonumber\\
&&\qquad=-\big((s-1)\theta_i^1(n)\theta^1_{s+t-1}(n)+(i-1)\theta_{s}^1(n)\theta^1_{i+t-1}(n)\big)P_{n-1}^{[n]}\nonumber\\
&&\qquad\quad-d(i+s-2)\theta^1_{i+s+t-2}(n)P_{n-1}^{[n]}+2(\phi_n)_*(\mh(i;s))(-
\cl_n)^t);\nonumber\\
&&(\phi_n)_*(\mj(i;s)(-
\cl_n)^t)\nonumber\\&&\qquad =-((s-1)\theta_i^1(n)\theta_{s+t-1}^0+(i-1)\theta_{s}^1(n)\theta^0_{i+t-1}(n))P_{n}^{[n]}\nonumber\\
&&\qquad\quad-d(i+s-2)\theta_{i+s-2}^0(n)P_{n}^{[n]}+d(\phi_n)_*(\mh(i;s))(-
\cl_n)^t);\nonumber\\
&&(\phi_n)_*(\mk(i;s)(-
\cl_n)^t)\nonumber\\
&&\qquad=-\big((s-1)\theta_i^0(n)\theta^1_{s+t-1}(n)+(i-1)\theta_{s}^0(n)\theta^1_{i+t-1}(n)\big)P_{n}^{[n]}\nonumber\\
\end{eqnarray}
Notice that by (\ref{svopn}) for $t\geq0,t+s\geq3,t+i\geq3$
\begin{eqnarray}\mh(i;s)(-
\cl_n)^t)&=&\big(\theta_i^1(n)(-\cl_n)^{s+t-2}(s-1)+\theta_{s}^1(n)(-\cl_n)^{i+t-2}(i-1)\big)HP_{n}^{[n]}\nonumber\\
&&+d(-\cl_n)^{i+s+t-3}(i+s-2)HP_{n}^{[n]};\nonumber
\\\nonumber
&=&\begin{cases}0,\qquad\qquad\qquad\qquad\qquad\qquad\qquad\text{ for }i+t,s+t\geq 4\\ d(i-1)\theta_s^1(n)H^2P_{n}^{[n]},\qquad\qquad\qquad\text{ for }i+t=3,s+t\geq 4 \\ d\big((i-1)\theta_s^1(n)+(s-1)\theta_i^1(n)\big)H^2P_{n}^{[n]},\text{ for }i+t=s+t=3\end{cases}.
\end{eqnarray}
Therefore for $t\geq 0,s+t\geq3,i+t\geq 3$ by induction assumption we have
\begin{eqnarray}&&-d^{-1}(d\bar{d}-2)^{-1}(\phi_n)_*(\tw(i;s) (-\cl_n)^t)\nonumber\\
&=&\frac1d(\phi_n)_*(\mh(i;s)(-\cl_n)^t))-\big((i-1)\theta_{i+t-3}^2(n)\theta_{s}^1(n)+(s-1)\theta_{s+t-3}^2(n)\theta_i^1(n)\big)\nonumber\\&=&0\nonumber
\end{eqnarray}
Now we are left for two cases: $t=0,i=2$,$s\geq3$ and $t=0,i=s=2$. 
We have 
\begin{eqnarray}&&(\phi_n)_*(\mf(2;s\geq3))\nonumber\\
&=&-(s-1)\theta_2^1(n)\theta^1_{s-1}(n)P_{n-1}^{[n]}-ds\theta^1_{s}(n)P_{n-1}^{[n]}+(\phi_n)_*(\mh(i;s))+d\theta_s^1(n)P_{n-1}^{[n]}\nonumber\\
&=&-(s-1)\big(\theta_2^1(n)\theta^1_{s-1}(n)+d\theta^1_{s}(n)\big)P_{n-1}^{[n]}+(\phi_n)_*(\mh(i;s)).\nonumber
\end{eqnarray}
By induction assumption and (\ref{Ag22tp2}) we have
\begin{eqnarray}&&\frac1d\big(\theta_i^1(n)\theta^1_{j}(n)+d\theta^1_{i+j-1}(n)\big)P_{n-1}^{[n]}\nonumber\\
&=&\bar{d}\left(\frac1d(\theta_i^0\theta_j^1(n)+\theta_j^0(n)\theta_i^1(n))+\theta_{i+j-1}^0(n)\right)P_n^{[n]}\nonumber\\
&&+(d\bar{d}-2)\theta_{i+j-2}^1(n)P_n^{[n]}-(d\bar{d}-2)\big(\theta_i^1(n)\theta_{j-2}^2(n)+\theta_{i-2}^2(n)\theta_{j}^1(n)\big)P_n^{[n]}.\nonumber
\end{eqnarray}

Therefore by a direct computation we have
 \begin{eqnarray}&&-(s-1)^{-1}d^{-1}(\phi_n)_*(\tw(2;s\geq 3))\nonumber\\
&=&\left(\frac1d\theta_2^1(n)\theta_{s-1}^1(n)+\theta_{s}^1(n)\right)P_{n-1}^{[n]}-\bar{d}\left(\frac1d(\theta_2^0(n)\theta_{s-1}^1(n)+\theta_{s-1}^0(n)\theta_2^1(n))+\theta_{s}^0(n)\right)P_n^{[n]}\nonumber\\
&&+(d\bar{d}-2)\theta_{s-1}^1(n)P_n^{[n]}-(d\bar{d}-2)\theta_{s-1}^1(n)P_n^{[n]}=0.\nonumber
\end{eqnarray}
Finally 
\[(\phi_n)_*(\mf(2;2))=(d\theta_2^1(n)+d\theta_2^1(n)-2d\theta^1_{2}(n))P_{n-1}^{[n]}=0.\]
By a direct computation we have
\[(\phi_n)_*(\tw(2;2))=-\bar{d}(\phi_n)_*(\mj(2,2)+\mk(2,2))=-\bar{d}\big(-2d\theta_{2}^0(n)+2d\theta_2^0(n)\big)P_{n}^{[n]}=0.\]
\end{proof}
\section{The proof of Lemma \ref{pnt41}.}\label{pl41}
Recall that we want to prove that in $H^*(S^{[n]},\bq)$ for all $i,j,i+k,j+k\geq2$
\begin{eqnarray}\label{Ag3tp1}&&\qquad(j-i)\big(\theta_{i+j+k-1}^0(n)-d\theta_{i+j+k-2}^1(n)\big)P_{n}^{[n]}=\nonumber\\
&&P_n^{[n]}\left[\frac{j-1}{d}\big(\theta_i^0(n)\theta_{j+k}^1(n)-\theta_i^1(n)\theta_{j+k}^0(n)\big)-\frac{i-1}{d}\big(\theta_j^0(n)\theta_{i+k}^1(n)-\theta_{j}^1(n)\theta_{i+k}^0(n)\big)\right]\nonumber\\&&+d(j-1)\big(\theta_i^1(n)\theta_{j+k-2}^2(n)-\theta_{i-2}^2(n)\theta_{j+k}^1(n)\big)P_n^{[n]}\nonumber\\
&&-d(i-1)\big(\theta_{j}^1(n)\theta_{i+k-2}^2(n)-\theta_{j-2}^2(n)\theta_{i+k}^1(n)\big)P_n^{[n]}\nonumber\\
\end{eqnarray}
In particular let $i,j\geq2,k=0$, then we have\small
\begin{eqnarray}\label{Ag3tp2}&&\qquad(j-i)\big(\theta_{i+j-1}^0(n)-d\theta_{i+j-2}^1(n)\big)P_{n}^{[n]}=\nonumber\\
&&P_n^{[n]}\left[\frac{i+j-2}{d}\big(\theta_i^0(n)\theta_{j}^1(n)-\theta_i^1(n)\theta_{j}^0(n)\big)+d(j+i-2)\big(\theta_i^1(n)\theta_{j-2}^2(n)-\theta_{i-2}^2(n)\theta_{j}^1(n)\big)\right]\nonumber\\\end{eqnarray}
\normalsize

Since both sides of (\ref{Ag3tp1}) are antisymmetric on $i,j$, we may assume $j>i$.  We do induction on $n$.  When $n=1$, (\ref{Ag3tp1}) holds by dimension reason.  

We will compare the pullbacks of LHS and RHS of (\ref{Ag3tp1}) to $\p(\mi_n)$ via $\psi_n:\p(\mi_n)\ra S^{[n+1]}$.
\subsection{The computation on the LHS and RHS}  
The LHS of (\ref{Ag3tp1}) is relatively simpler.  By (\ref{anpn}), (\ref{svopn}) and (\ref{bnpnh}) for $i\geq 2,j\geq3$ on $\p(\mi_n)$ we have
\begin{eqnarray}\label{ex4lhs1}&&(j-i)\big(\theta_{i+j+k-2}^1(n+1)P_{n+1}^{[n+1]}-\theta_{i+j+k-2}^1(n)P_n^{[n]}(\cl_n+dH)\big)
\nonumber\\
&=&-(j-i)A_{i+j+k-2}(n)HP_n^{[n]}(\cl_n+dH)\nonumber\\
&=&(j-i)(i+j+k-3)(-\cl_n)^{i+j+k-3}HP_n^{[n]}
\end{eqnarray}
and
\begin{eqnarray}\label{ex4lhs}&&(j-i)\big(\theta_{i+j+k-1}^0(n+1)P_{n+1}^{[n+1]}-\theta_{i+j+k-1}^0(n)P_n^{[n]}(\cl_n+dH)\big)
\nonumber\\
&=&-(j-i)(A_{i+j+k-1}(n)+B_{i+j+k-1}(n)H)P_n^{[n]}(\cl_n+dH)\nonumber\\
&=&(j-i)(i+j+k-2)(-\cl_n)^{i+j+k-2}P_n^{[n]}-d(j-i)(i+j+k-2)(-\cl_n)^{i+j+k-3}HP_n^{[n]}\nonumber\\
&&-3(j-i)\binom{i+j+k-2}{2}(-\cl_n)^{i+j+k-4}HP_n^{[n]}(\cl_n+dH)\nonumber\\
&&-d(j-i)\binom{i+j+k-3}2(-\cl_n)^{i+j+k-3}HP_n^{[n]}\nonumber\\
&=&(j-i)(i+j+k-2)(-\cl_n)^{i+j+k-2}P_n^{[n]}
\nonumber\\
&&-3(j-i)\binom{i+j+k-2}{2}(-\cl_n)^{i+j+k-4}HP_n^{[n]}(\cl_n+dH)\nonumber\\
&&-d(j-i)\left(\binom{i+j+k-2}2+1\right)(-\cl_n)^{i+j+k-3}HP_n^{[n]}\nonumber\\
\end{eqnarray}
Now we deal with the RHS of (\ref{Ag3tp1}).
We can compute that
\begin{eqnarray}\label{strhs4}&&\big(\theta_i^1(n+1)\theta_{j+k-2}^2(n+1)-\theta_{i-2}^2(n+1)\theta_{j+k}^1(n+1)\big)P_{n+1}^{[n+1]}\nonumber\\
&=&(\theta_i^1(n)\theta_{j+k-2}^2(n)-\theta_{i-2}^2(n)\theta_{j+k}^1(n))P_{n}^{[n]}(\cl_n+dH)+\mr(i;j+k),\nonumber\\
\end{eqnarray}
where for $i,s\geq 2$
\begin{eqnarray}
\mr(i;s)&:=&\big(\theta_{s-2}^2(n)(-\cl_n)^{i-1}(i-1)-\theta_{i-2}^2(n)(-\cl_n)^{s-1}(s-1)\big)HP_n^{[n]}\nonumber\\
&&-d\big(\theta_{s-2}^2(n)(-\cl_n)^{i-2}(i-1)-\theta_{i-2}^2(n)(-\cl_n)^{s-2}(s-1)\big)H^2P_n^{[n]};\nonumber\\
\end{eqnarray}
Moreover for $i+s\geq5$ we have
\begin{eqnarray}\label{eprme4}
\mr_(i;s)&=&\big(\theta_{s-2}^2(n)(-\cl_n)^{i-1}(i-1)-\theta_{i-2}^2(n)(-\cl_n)^{s-1}(s-1)\big)HP_n^{[n]}\nonumber\\
\end{eqnarray}

For the rest terms of RHS we have
\begin{eqnarray}&&P_{n+1}^{[n+1]}\left[\frac{j-1}{d}\big(\theta_i^0(n+1)\theta_{j+k}^1(n+1)-\theta_i^1(n+1)\theta_{j+k}^0(n+1)\big)\right.\nonumber\\&&\left.-\frac{i-1}{d}\big(\theta_j^0(n+1)\theta_{i+k}^1(n+1)-\theta_{j}^1(n+1)\theta_{i+k}^0(n+1)\big)\right]\nonumber\\
&=&P_n^{[n]}(\cl_n+dH)\left[\frac{j-1}{d}\big(\theta_i^0(n)\theta_{j+k}^1(n)-\theta_i^1(n)\theta_{j+k}^0(n)\big)\right.\nonumber\\&&\left.-\frac{i-1}{d}\big(\theta_j^0(n)\theta_{i+k}^1(n)-\theta_{j}^1(n)\theta_{i+k}^0(n)\big)\right]\nonumber\\
&&-\frac{j-1}d(\mg(i;j+k)-\mg(j+k;i))+\frac{i-1}d(\mg(j;i+k)-\mg(i+k;j)),\nonumber\\
\end{eqnarray}  
where for $i,s\geq 2$
\begin{eqnarray}\label{defgr}
\mg(i;s)&:=&\big(\theta_s^1(n)A_i(n)+\theta_s^1(n)B_i(n)H+\theta_i^0(n)A_s(n)H\big)P_n^{[n]}(\cl_n+dH)\nonumber\\
&=&-(i-1)\theta_s^1(n)(-\cl_n)^{i-1}P_n^{[n]}+d(i-1)\theta_s^1(n)(-\cl_n)^{i-2}HP_n^{[n]}\nonumber\\
&&-(s-1)\theta_i^0(n)(-\cl_n)^{s-1}HP_n^{[n]}+d(s-1)\theta_i^0(n)(-\cl_n)^{s-2}H^2P_n^{[n]}\nonumber\\
&&+3d\binom{i-1}{2}(-\cl_n)^{i+s-3}HP_n^{[n]}\nonumber\\
&&+d\binom{i-2}{2}\theta_s^1(n)(-\cl_n)^{i-2}HP_n^{[n]}\nonumber\\
&=&-(i-1)\theta_s^1(n)(-\cl_n)^{i-1}P_n^{[n]}+3d\binom{i-1}{2}(-\cl_n)^{i+s-3}HP_n^{[n]}\nonumber\\
&&-(s-1)\theta_i^0(n)(-\cl_n)^{s-1}HP_n^{[n]}+d(s-1)\theta_i^0(n)(-\cl_n)^{s-2}H^2P_n^{[n]}\nonumber\\
&&+d\left(\binom{i-1}{2}+1\right)\theta_s^1(n)(-\cl_n)^{i-2}HP_n^{[n]}\nonumber\\
\end{eqnarray}
The two equations in (\ref{defgr}) are because of (\ref{anpn}), (\ref{svopn}) and (\ref{bnpnh}).

\subsection{The proof.}
For $i,s\geq 2$, $i+s\geq5$ define
\begin{eqnarray}
\rl_1(i;s)&:=&d(i+s-2)(-\cl_n)^{i+s-2}P_n^{[n]}
\nonumber\\
&&+3d\binom{i+s-2}{2}(-\cl_n)^{i+s-3}HP_n^{[n]}\nonumber\\
&&+d^2\binom{i+s-1}2(-\cl_n)^{i+s-3}HP_n^{[n]}=\rl_1(s;i)\nonumber\\
\rr_1(i;s)&:=&\mg(i;s)-\mg(s;i)-d^2\mr(i;s).\nonumber\\
\end{eqnarray}

Then by induction assumption
\[LHS-RHS=\frac{j-i}d\rl_1(i;j+k)+\big(\frac{j-1}d\rr_1(i;j+k)-\frac{i-1}d\rr_1(j;i+k)\big).\]
Therefore by Lemma \ref{omep} it suffices to show that for any $k,t\geq0$, $j>i\geq2$, 
\begin{equation}(\phi_n)_*\big((LHS-RHS)(-\cl_n)^tH^k\big)=0.\end{equation}
Since $j>i\geq2$, by Lemma \ref{pnt2} (4) we have
\begin{eqnarray}
\frac{j-i}d\rl_1(i;j+k)\cdot H&=&(j-i)(i+j+k-2)(-\cl_n)^{i+j+k-2}HP_n^{[n]};\nonumber\\
\frac{j-1}d\rr_1(i;j+k)\cdot H&=&\frac{j-1}d\big((i-1)-(j+k-1)\big)(-\cl_n)^{i+j+k-2}HP_n^{[n]}\nonumber\\
&=&\frac{j-1}d\big(2(i-1)-(i+j+k-2)\big)(-\cl_n)^{i+j+k-2}HP_n^{[n]}.
\end{eqnarray}
Therefore $$\frac{j-i}d\rl_1(i;j+k)\cdot H=-\frac{j-1}d\rr_1(i;j+k)\cdot H+\frac{i-1}d\rr_1(j;i+k)\cdot H.$$

We only need to show that for any $t\geq0$, $j>i\geq2$, 
\begin{equation}(\phi_n)_*\big((LHS-RHS)(-\cl_n)^t\big)=0.\end{equation}
For $t\geq 0$ and $i,s\geq 2$ we have
\begin{eqnarray}\label{mgva}&&(\phi_n)_*(\mg(i;s)(-\cl_n)^t)\nonumber\\
&=&-(i-1)\theta_s^1(n)\theta_{i+t-1}^0(n)P_n^{[n]}+d(s-1)\theta_i^0(n)\theta_{s+t-2}^2(n)P_n^{[n]}\nonumber\\
&&-(s-1)\theta_i^0(n)\theta_{s+t-1}^1(n)P_n^{[n]}+3d\binom{i-1}{2}\theta_{i+s+t-3}^1(n)P_n^{[n]}\nonumber\\
&&+d\left(\binom{i-1}{2}+1\right)\theta_s^1(n)\theta_{i+t-2}^1(n)P_n^{[n]}\nonumber\\
\end{eqnarray}
Notice that by Lemma \ref{pnt31} (1) we have for $t\geq0$, $i,s\geq2$
\begin{eqnarray}&&\left(\binom{i-1}{2}+1\right)\theta_s^1(n)\theta_{i+t-2}^1(n)P_n^{[n]}\nonumber\\
&=&\begin{cases}
-d\left(\binom{i-1}{2}+1\right)\theta_{i+s+t-3}^1(n)P_n^{[n]},\text{ for }i+t\geq 4\\ 0,\qquad\qquad\text{ for }i+t=2,3\end{cases}
\end{eqnarray}
Therefore for $t\geq0$, $t+s,t+i\geq3$, $i,s\geq 2$ we have
\begin{eqnarray}&&\left(\binom{i-1}{2}+1\right)\theta_s^1(n)\theta_{i+t-2}^1(n)P_n^{[n]}\nonumber\\
&=&
-d\left(\binom{i-1}{2}+1\right)\theta_{i+s+t-3}^1(n)P_n^{[n]}+d(i-1)\theta_{i+t-3}^2(n)\theta_s^1(n)P_n^{[n]}.\nonumber\end{eqnarray}
Hence by Lemma \ref{pnt2} (1) and Lemma \ref{pnt31} (1) for $t\geq 0$, $s+t,i+t\geq 3$ and $i,s\geq 2$ we have
\begin{eqnarray}&&(\phi_n)_*(\big(\mg(i;s)-\mg(s;i)\big)(-\cl_n)^t)\nonumber\\
&=&(i-1)\big(\theta_s^0(n)\theta_{i+t-1}^1(n)-\theta_s^1(n)\theta_{i+t-1}^0(n)\big)P_n^{[n]}\nonumber\\
&&-(s-1)\big(\theta_i^0(n)\theta_{s+t-1}^1(n)-\theta_i^1(n)\theta_{s+t-1}^0(n)\big)P_n^{[n]}\nonumber\\
&&+3d\left(\binom{i-1}{2}-\binom{s-1}{2}\right)\theta_{i+s+t-3}^1(n)P_n^{[n]}\nonumber\\
&&-d^2\left(\binom{i-1}{2}-\binom{s-1}2\right)\theta_{i+s+t-3}^1(n)P_n^{[n]}\nonumber\\
&&+d^2\big((i-1)\theta_{i+t-3}^2(n)\theta_s^1(n)-(s-1)\theta_{s+t-3}^2(n)\theta_i^1(n)\big)P_n^{[n]}.\nonumber
\end{eqnarray}
Hence
\begin{eqnarray}&&(\phi_n)_*(\rr_1(i;s)(-\cl_n)^t)=(\phi_n)_*(\big(\mg(i;s)-\mg(s;i)-d^2\mr(i;s)\big)(-\cl_n)^t)\nonumber\\
&=&(i-1)\big(\theta_s^0(n)\theta_{i+t-1}^1(n)-\theta_s^1(n)\theta_{i+t-1}^0(n)\big)P_n^{[n]}\nonumber\\
&&-(s-1)\big(\theta_i^0(n)\theta_{s+t-1}^1(n)-\theta_i^1(n)\theta_{s+t-1}^0(n)\big)P_n^{[n]}\nonumber\\&&+3d\left(\binom{i-1}{2}-\binom{s-1}{2}\right)\theta_{i+s+t-3}^1(n)P_n^{[n]}\nonumber\\
&&-d^2\left(\binom{i-1}{2}-\binom{s-1}2\right)\theta_{i+s+t-3}^1(n)P_n^{[n]}\nonumber\\
&&+d^2\big((i-1)\theta_{i+t-3}^2(n)\theta_s^1(n)-(s-1)\theta_{s+t-3}^2(n)\theta_i^1(n)\big)P_n^{[n]}\nonumber\\&&-d^2\big((i-1)\theta_{i+t-1}^1(n)\theta_{s-2}^2(n)-(s-1)\theta_{s+t-1}^1(n)\theta_{i-2}^2(n)P_n^{[n]}.
\nonumber
\end{eqnarray}
By induction assumption we have
\begin{eqnarray}\label{rr1n}&&(\phi_n)_*(\rr_1(i;s)(-\cl_n)^t)\nonumber\\
&=&d(i-s)\big(\theta_{i+s+t-2}^0(n)-d\theta_{i+s+t-3}^1(n)\big)P_n^{[n]}\nonumber\\
&&+3d\left(\binom{i-1}{2}-\binom{s-1}{2}\right)\theta_{i+s+t-3}^1(n)P_n^{[n]}\nonumber\\
&&-d^2\left(\binom{i-1}{2}-\binom{s-1}2\right)\theta_{i+s+t-3}^1(n)P_n^{[n]}\nonumber\\
\end{eqnarray}
Notice that we have
\[(j-1)(i-(j+k))-(i-1)(j-(i+k))=-(j-i)(j+k+i-2);\]
and for any $v,u,w$
\begin{eqnarray}\label{dff4}&&(j-1)\left(\frac{(i-1)(i-1+v)}{2}-\frac{(j+w-1)(j+w-1+u)}{2}\right)\nonumber\\
&&-(i-1)\left(\frac{(j-1)(j-1+v)}{2}-\frac{(i+w-1)(i+w-1+u)}{2}\right)\nonumber\\
&=&-(j-i)\frac{(i+j+w-2)(i+j+w+u-2)}{2}\nonumber\\
\end{eqnarray}
Hence for $t\geq 0,i+t\geq 3$
\begin{eqnarray}&&(\phi_n)_*\left(\big(\frac{j-1}d\rr_1(i;j+k)-\frac{i-1}d\rr_1(j;i+k)\big)(-\cl_n)^t\right)\nonumber\\
&=&-(j-i)(i+j+k-2)\big(\theta_{i+j+k+t-2}^0(n)-d\theta_{i+j+k+t-3}^1(n)\big)P_n^{[n]}\nonumber\\
&&-3(j-i)\binom{i+j+k-2}{2}\theta_{i+j+k+t-3}^1(n)P_n^{[n]}\nonumber\\
&&+(j-i)d\binom{i+j+k-2}{2}\theta_{i+j+k+t-2}^1(n)P_n^{[n]}\nonumber\\
&=&-(\phi_n)_*\big(\frac{j-i}d\rl_1(i;j+k)(-\cl_n)^t\big),\nonumber
\end{eqnarray}
where the last equation is because $\binom{i+j+k-2}{2}+i+j+k-2=\binom{i+j+k-1}{2}$.

Now we are only left to the two cases: $t=0,i=2,k\geq 1,j\geq 3$ and $t=k=0,i=2,j\geq 3$.

Easy to see we always have $\rr_1(i;s)=-\rr_1(s;i)$.  By (\ref{mgva}) we have
\begin{eqnarray}&&(\phi_n)_*(\rr_1(2;s\geq 3))=(\phi_n)_*(\mg(2;s\geq 3)-\mg(s\geq3;2)-d^2\mr(2,s\geq3))\nonumber\\
&=&(s-1)(\theta_2^1(n)\theta_{s-1}^0(n)-\theta_2^0(n)\theta_{s-1}^1(n))P_n^{[n]}-d\theta_s^0(n)P_n^{[n]}\nonumber\\
&&-3d\binom{s-1}{2}\theta_{s-1}^1(n)P_n^{[n]}-d\left(\binom{s-1}{2}+1\right)\theta_2^1(n)\theta_{s-2}^1(n)P_n^{[n]}\nonumber\\
&&+d^2(s-1)\theta_{s-1}^1(n)P_n^{[n]}.\nonumber
\end{eqnarray}
By induction assumption and (\ref{Ag3tp2}) we have
\begin{eqnarray}&&(s-1)(\theta_2^1(n)\theta_{s-1}^0(n)-\theta_2^0(n)\theta_{s-1}^1(n))P_n^{[n]}
\nonumber\\
&=&-(s-3)d\big(\theta_s^0(n)-d\theta_{s-1}^1(n)\big)P_n^{[n]}-d^2(s-1)\theta_{s-1}^1(n).
\end{eqnarray}
Therefore
\begin{eqnarray}\label{rr1s}&&(\phi_n)_*(\rr_1(2;s\geq 3))\nonumber\\
&=&d(2-s)\big(\theta_s^0(n)-d\theta_{s-1}^1(n)\big)P_n^{[n]}-d^2\theta_{s-1}^1(n)P_n^{[n]}\nonumber\\
&&-3d\binom{s-1}{2}\theta_{s-1}^1(n)P_n^{[n]}-d\left(\binom{s-1}{2}+1\right)\theta_2^1(n)\theta_{s-2}^1(n)P_n^{[n]}\nonumber\\
&=&d(2-s)\big(\theta_s^0(n)-d\theta_{s-1}^1(n)\big)P_n^{[n]}\nonumber\\
&&+3d\left(\binom{i-1}{2}-\binom{s-1}{2}\right)\theta_{i+s+t-3}^1(n)P_n^{[n]}\nonumber\\
&&-d^2\left(\binom{i-1}{2}-\binom{s-1}2\right)\theta_{s-1}^1(n)P_n^{[n]}\nonumber\\
&&-d(s-1)\theta_{s-3}^2(n)\theta_2^1(n)P_n^{[n]}.\nonumber\\
\end{eqnarray}
By (\ref{rr1s}) we see that (\ref{rr1n}) also holds for $(\phi_n)_*(\rr_1(2;s\geq 4))$, and hence we are done for cases $t=0,i=2,k\geq 1,j\geq3$ and $t=k=0,i=2,j\geq4$.  

We are left to only one case: $t=k=0,i=2,j=3$.  By (\ref{rr1s}) we have
\begin{eqnarray}&&(\phi_n)_*\left(\big(\frac{2}d\rr_1(2;3)-\frac{1}d\rr_1(3;2)\big)\right)=\frac3d(\phi_n)_*\rr_1(2;3)\nonumber\\
&=&-\frac3d\big(d\theta_3^0(n)+(3d+2d^2)\theta_2^1(n)\big)\nonumber\\
&=&-(\phi_n)_*\big(\frac{1}d\rl_1(2;3)\big),\nonumber
\end{eqnarray}

We have finished the proof.


Yao Yuan\\
Beijing National Center for Applied Mathematics,\\
Academy for Multidisciplinary Studies, \\
Capital Normal University, 100048, Beijing, China\\
E-mail: 6891@cnu.edu.cn.
\end{document}